\documentclass[twoside]{article}
\usepackage[a4paper]{geometry}
\usepackage[utf8]{inputenc}
\usepackage[T1]{fontenc} 
\usepackage{RR}
\usepackage{hyperref}
\usepackage{amsfonts,amsthm,latexsym,amssymb,amscd,epsf,stmaryrd}
\usepackage{amsmath,,amsthm,booktabs,color,epsfig,graphicx,url}
\theoremstyle{plain}

\newtheorem{proposition}{Proposition}

\theoremstyle{definition}

\newtheorem{example}{Example}

\def\R{\mathbb{R}}
\def\N{\mathbb{N}}
\def\S{\mathcal{S}}

\def\ort{\mathbb{O}}
\newcommand{\sqb}[2]{\{ #1,\dots, #2 \}}
\newcommand{\ten}[1]{\mathbf{#1}}
\renewcommand{\vec}[1]{\mathbf{#1}}
\newcommand{\ignore}[1]{} 
\renewcommand{\top}{\textsc{T}}

\usepackage{graphicx}
\usepackage{subfig}

\usepackage{float}
\usepackage[export]{adjustbox}
\usepackage{pgfplots,relsize,siunitx}
\usepackage{tikz}
\usetikzlibrary{matrix}

\usepackage[left, mathlines]{lineno}
\runninglinenumbers
\nolinenumbers
\usepackage{booktabs}

\usepackage{array}
\usepackage{makecell}
\usepackage{url}
\RRNo{9429}
\RRdate{\today}
\RRauthor{
	Olivier Coulaud\thanks[sfn]{HiePACS team, Inria Bordeaux-Sud-Ouest, France}
  \and
Alain Franc\thanks{BIOGECO, INRAE, Univ. Bordeaux, 33610 Cestas, France}\thanks[fn]{
Pleiade team - Inria\textbackslash INRAE Bordeaux-Sud-Ouest, France}%
\and Martina Iannacito\thanksref{sfn}}
\authorhead{Coulaud \& Franc \& Iannacito}
\RRtitle{Extension de l'analyse des correspondances à des donnéees multi-dimensionnelles : un point de vue géométrique}
\RRetitle{Extension of Correspondence Analysis to multiway data-sets through High Order SVD: a geometric framework}
\titlehead{Extension of CA to multiway data-sets through HOSVD}
\RRmotcle{Analyse des correspondances, modèle de Tucker, HOSVD, tenseurs, relation barycentrique, nuages de points}
\RRkeyword{Correspondence analysis,
	Tucker model,
	HOSVD,
	tensors,
	barycentric relation,
	point clouds}
\RRresume{Ce document présente une extension de l'analyse des correspondances aux tenseurs par la décomposition en valeurs singulières d'ordre élevé (HOSVD) d'un point de vue géométrique. L'analyse des correspondances est un outil bien connu, développé à partir de l'analyse en composantes principales, pour étudier les tables de contingence. Différentes extensions algébriques de l'analyse des correspondances  aux tables à voies multiples  ont été proposées au fil des ans. En nous appuyant sur le modèle de Tucker et la HOSVD, nous proposons d'associer à chaque mode d'un tenseur un nuage de points. Nous établissons un lien entre les cordonnées de ces différents nuages. Une telle relation est classique en Analyse Factorielle des Correspondances (AFC) pour justifier la projection simultanée des profils lignes et profils colonnes d'une table de contingence (d'où le nom de \emph{correspondance}). Nous étendons une telle relation barycentrique aux liens entre les nuages de points associés aux différents modes de l'Analyse Factorielle des Correspondances Multiple d'un tenseur, construite via la HOSVD avec les métriques de l'AFC.}

\RRabstract{This paper presents an extension of Correspondence Analysis (CA) to tensors through High Order Singular Value Decomposition (HOSVD) from a geometric viewpoint. Correspondence analysis is a well-known tool, developed from principal component analysis, for studying contingency tables. Different algebraic extensions of CA to multi-way tables have been proposed over the years, nevertheless neglecting its geometric meaning. Relying on the Tucker model and the HOSVD, we propose a direct way to associate with each tensor mode a point cloud. We prove that the point clouds are related to each other. Specifically using the CA metrics we show that the barycentric relation is still true in the tensor framework. Finally two data sets are used to underline the advantages and the drawbacks of our strategy with respect to the classical matrix approaches.}

\RRprojets{HiePACS and Pleiade}
\RCBordeaux 

\begin{document}
\makeRR   
\clearpage

\section{Introduction}
	Dimension reduction has been and still is one of the most widely used method to address modeling questions on data sets living in large dimensional spaces~\cite[with further detailed references within]{Izenman2008,Wang2012}. One of the most popular and used method is Principal Component Analysis (PCA)~\cite{Pearson1901,Hotelling1933,Gabriel1971,Atchley1975, Kendall1975, morrison1976, chambers1977,Jolliffe2002,Jolliffe2016}. It has three guises
	\begin{enumerate}
		\item an algebraic framework, working with vector spaces and matrix algebra, based on the Singular Value Decomposition (SVD) of the data matrix, see, e.g.,~\cite{Mardia1979} 
		\item a geometric framework, by associating a point cloud with the matrix representing the data set (one point per row, one dimension per feature), 
		see, e.g.,~\cite{LMF82,Saporta1990}
		\item a statistical modeling approach, see, e.g.,~\cite{Anderson1958,Rao1973}.
	\end{enumerate}
	PCA has been extended to a diversity of situations, like Correspondence Analysis (CA)~\cite{benzecri1973,Hill1974,Greenacre1977,greenacre1978,gauch_1982,gifi1990,Greenacre1984} for the analysis of contingency tables with metrics associated with $\chi^2$ distances~\cite{Greenacre1984,Greenacre2007}. From an algebraic perspective, CA is the PCA of a contingency table with the metric associated with the inverse of row and column marginals. CA has been explained in a geometric framework as well~\cite{Cordier1965, Benzecri1969}. Indeed, as for PCA, in the geometric view of CA each row of a contingency table corresponds to the coordinates of a point in a specific vector space. Since a similar argument holds for both rows and columns, a contingency table is naturally associated with a two point clouds. Thanks to the underlying dimension reduction techniques, CA makes possible the simultaneous visualization and interpretation of the two point clouds in a low dimension space where a specific barycentric relation links the two objects. The dictionary between the geometric and algebraic framework relies on the selection of weights and metrics in the spaces where the point clouds live, see~\cite{LMF82,Saporta1990}. 
	
	\par Multiway data have appeared in a wide range of domains, requesting a further development of these analysis techniques. PCA and associated methods have been extended algebraically to dimension reduction in multiway array, see, e.g.,~\cite{Kroonenberg1983,Bolasco1989,Franc1992, Lathauwer2000a,Lathauwer2000b,Kroonenberg2008,Kolda2009} and many references therein, with a link with tensor algebra~\cite{Franc1989}. These are algebraic extensions, relying on numerical computations based on elementary operations in tensor algebra in the same way that PCA and CA rely on numerical linear algebra. 
	Starting from the generalization of PCA to multiway data through the Tucker model~\cite{Tucker1966}, with High Order Singular Value Decomposition (HOSVD)~\cite{Lathauwer2000a}, an algebraic development of CA to MultiWay Correspondence Analysis (MWCA) follows naturally. Indeed MWCA is the HOSVD of a multiway contingency table associated with metric of the marginal inverse per each mode, see~\cite{Franc1992}. Our work aims to study the geometric perspective of MWCA, which is still poorly investigated. More precisely, we show that a point cloud is associated with each mode of a tensor in MWCA in the same way that a point cloud is associated with either rows or columns of a matrix. Finally, the correspondence between two point clouds based on scaling and computing barycenters holding in the classical CA is generalized to $d$ point clouds in MWCA.
	
	\par The remainder of this paper is organized as follows. Section~\ref{sec:2} introduces all the preliminary tensor definitions, spaces and operations, clarifying the chosen notations. In Section~\ref{sec:3} the algebraic link between the principal components is presented. In particular, Subsection~\ref{sec:4} extends the results of the previous subsections in a space with generic metric. Section~\ref{sec:5} presents the barycentric relation characterizing Correspondence Analysis in the tensor case. Finally we highlight the different outcome on two data sets using the previous theoretical results.

\section{Preliminaries\label{sec:2}}
\subsection{Notations}
	Here real numbers and integers are denoted by small Latin letters, vectors by boldface Latin letters, matrices by capital Latin letters and tensors by boldface capital Latin letters. $\ort(m\times n)$ denotes the set of all the real orthogonal matrices of $m$ rows and $n$ columns. Let us have $d$ finite dimensional vector spaces $E_1,\dots,E_d$ on a same field $\R$. Let $\vec{A}$ be a tensor in $E_1 \otimes \dots \otimes E_d$ with $\mathrm{dim}\: E_\mu=n_\mu$ for $\mu \in \sqb{1}{d}$. The integer $d$ is called the order of the tensor $\vec{A}$, each vector space $E_\mu$ 	is called a mode and $n_\mu$ is the dimension of the tensor $\vec{A}$ for mode $\mu$. We denote as well $\vec{A} \in \R^{n_1 \times \dots 
	\times n_d}$. The matricization or unfolding of $\vec{A}$ with respect to the mode $\mu$ denoted by $A^{(\mu)}$ is a matrix of $n_\mu$ rows and $n_{\ne \mu}$ columns with $n_{\ne \mu} = \prod_{\alpha\ne \mu}n_\alpha$, see~\cite[Definition~1]{Lathauwer2000a}. Notice that the matricization with respect to mode $\mu$ maps the $(i_1,\dots,i_d)$-th tensor element to the $(i_\mu, \overline{i_1,\dots,i_{\mu-1}, i_{\mu+1},\dots, i_d})$-th matrix element with
	\begin{equation}
	\label{eq2:1}
	\overline{i_1,\dots,i_{\mu-1}, i_{\mu+1},\dots, i_d} = 1 + \sum_{\substack{ \alpha = 1\\ \alpha\ne \mu}}^{d}(i_\alpha - 1)m_\alpha\qquad\text{with}\qquad m_\alpha = \prod_{\substack{ \beta = 1\\ \beta\ne \mu}}^{\alpha-1}n_\beta,
	\end{equation}
	see~\cite{Kolda2009}. Finding an approximation of Tucker model is facilitated by using an elementary operations on tensors: the Tensor-Times-Matrix product (TTM). Let $\vec{A} \in E_1 \otimes \dots \otimes E_d$, and $M_\mu \in \mathcal{L}(E_\mu,E_\mu) \simeq E_\mu \otimes E_\mu$ for any mode $\mu$. Let $\vec{A}$ be an elementary tensor $\vec{A} = \vec{a}_1 \otimes \dots \otimes \vec{a}_d$ with $\vec{a}_\mu\in E_{\mu}$. Then, TTM product of $(M_1, 
	\dots,M_d)$ with $\vec{A}$, denoted $(M_1,\dots,M_d)\vec{A}$, is the tensor $M_1\vec{a}_1 \otimes \dots \otimes M_d\vec{a}_d$. This is extended to any tensor by linearity, see, e.g.,~\cite{Lathauwer2000a}. Let $\vec{A}$ be a tensor of $E_1\otimes \dots\otimes E_d$ and let $(M_1,\dots, M_d)$ be a $d$-tuple of matrices of compatible dimension with $\vec{A}$. If $\vec{B} = (M_1,\dots, M_d)\vec{A}$, then the $\mu$-matricization of $\vec{B}$ is
	\begin{equation}
  	\label{rk:1}
  	B^{(\mu)} = M_\mu A^{(\mu)} (M_{d}\otimes_K\dots\otimes_K M_{\mu+1}\otimes_K M_{\mu-1}\otimes_K\dots\otimes_K M_1)^\top
	\end{equation} 
	for every $\mu\in\sqb{1}{d}$ and, $\otimes_K$ the Kronecker product, we refer to~\cite[Proposition 3.7]{SAND2006-2081} for further details.
	\par Let $N_\mu=M_\mu^2$ be a Symmetric Positive Definite (SPD) matrix defining an inner product on $E_\mu$ by $\langle \vec{a},\vec{a}'\rangle_{M_\mu} = \langle M_\mu\vec{a}, M_\mu\vec{a}'\rangle = \langle N_\mu\vec{a}, \vec{a}'\rangle = \langle \vec{a}, N_\mu\vec{a}'\rangle$. It is extended to the whole spaces by linearity. This induces an inner product on $E_1 \otimes \dots \otimes E_d$ on elementary tensors by $\langle \vec{a}_1 \otimes \dots \otimes \vec{a}_d \; , \; \vec{a}'_1 \otimes \dots \otimes \vec{a}'_d\rangle_{\textsc{m}_1 \otimes \dots \otimes \textsc{m}_d} = 
	\langle M_1\vec{a}_1 \otimes \dots \otimes M_d\vec{a}_d \; , \; M_1\vec{a}'_1 \otimes \dots \otimes M_d\vec{a}'_d\rangle$. Let us denote by $\S = (\bigotimes_\mu E_\mu,\bigotimes_\mu\mathbb{I}_\mu)$ the tensor space $E_1 \otimes \dots \otimes E_d$ endowed with standard inner product with $\mathbb{I}_\mu$ being the identity matrix on $E_\mu$, and by $\S_\textsc{m} = (\bigotimes_\mu E_\mu, \bigotimes_\mu M_\mu)$ the same tensor space when endowed with inner product induced by the matrices $M_\mu$. We will use the observation that the map
	\[
	\nu: \S_{\textsc{m}} \rightarrow \S \qquad\text{with}\qquad \nu\bigl(\bigotimes_\mu \vec{a}_\mu\bigr) = \bigotimes_\mu M_\mu\vec{a}_\mu 
	\]
	is an isometry, because $\|\bigotimes_\mu M_\mu\vec{a}_\mu \| = \prod_\mu \|M_\mu\vec{a}_\mu\| = \prod_\mu \|\vec{a}_\mu\|_{\textsc{m}_\mu} = 
	\|\bigotimes_\mu \vec{a}_\mu \|_{\textsc{m}_1 \otimes \dots \otimes \textsc{m}_d}$ where $\|\cdot\|$ denotes the Frobenius norm. This can be extended to the whole tensor space $E_1 \otimes \dots \otimes E_d$ by linearity (see~\cite{Franc1992} for details). The isometry can be written as $\nu(\vec{A})= (M_1,\dots,M_d)\vec{A}$.

\subsection{Tucker model\label{sec2:5}} 
	The Tucker decomposition~\cite{Tucker1966, Kroonenberg1983,Kapteyn1986,Kroonenberg2008, Kolda2009} of tensor $\ten{A}\in E_1\otimes \cdots\otimes E_d$ is
	\begin{equation}
	\label{eq2:2}
		\vec{A} = \sum_{i_1=1}^{r_1} \dots \sum_{i_d=1}^{r_d}\, C_{i_1\dots i_d}\, \vec{u}^{1}_{i_1} \otimes \dots \otimes \vec{u}^{d}_{i_d}
	\end{equation}
	where the array of the $(C_{i_1\dots i_d})_{i_1,\dots,i_d}$ is the \emph{core tensor} of $\vec{A}$ and $(\vec{u}^{\mu}_{i_\mu})_{i_\mu}$ is an orthonormal basis of $U_\mu \subseteq E_\mu$ with $r_\mu$ minimal for $i_\mu\in\sqb{1}{r_\mu}$ and $\mu\in\sqb{1}{d}$. So if $\ten{A}$ belongs to $U_1\otimes \dots\otimes U_d$ with $U_\mu\subseteq E_\mu$ and $\text{dim}\,U_\mu = r_\mu$, then $\vec{r}=(r_1, \dots, r_d)$ is the \emph{multilinear rank} of $\ten{A}$~\cite{Lathauwer2000a}. In a synthetic way Tucker model is expressed with the TTM product as $\ten{T} = (U_1,\dots U_d)\ten{C}$. We identify each Tucker subspace with its orthonormal basis, denoting $U_\mu\in\ort(n_\mu \times r_\mu)$ for every mode $\mu\in\sqb{1}{ d}$. 
	
	\par Starting from this Tucker model, we formulate an approximation problem as follows. Given a tensor $\vec{A}$, its best Tucker approximation at multilinear rank $\vec{r}$ is the tensor $\vec{T}$ of multi-linear rank $\vec{r}$ such that $\|\vec{A}-\vec{T}\|$ is minimal. This is a natural extension of PCA because the unknowns are the spaces $(U_\mu)_\mu$ under constraints of dimensions. Historically speaking, finding a solution to Tucker best approximation has a long history which can be found, e.g., in~\cite{Tucker1966,Franc1992, Kolda2009, Grasedick2010}. There is no known algorithm yielding the best solution of the approximation problem, although several algorithms provide good quality results, known nowadays as High Order Singular Value Decomposition (HOSVD)~\cite{Lathauwer2000a}, its Truncated version (T-HOSVD)~\cite{StHOSVD} and High Order Orthogonal Iterations (HOOI)~\cite{Lathauwer2000b, Kolda2009}. The seminal paper for what is now called HOOI has been published as Tuckals3 for $3$-order tensors is~\cite{Kroonenberg1980} (see as well~\cite{Kroonenberg2008}). The extension to four mode tensors has been made by Lastovicka in~\cite{Lastovicka1981}, and more generally to $d$-order multi-arrays by a group in Groningen in 1986~\cite{Kapteyn1986}. These works used Kronecker product as an algebraic framework, and those results have been put into a common framework of tensor algebra in~\cite{Franc1992}. Both approaches (decomposition and best approximation) have been popularized by two papers by de Lauthauwer et al. in 2000, who derived HOSVD~\cite{Lathauwer2000a} and HOOI~\cite{Lathauwer2000b} relying on matricization and matrix algebra. Matricization, called unfolding as well, is building a matrix with one mode in row, and a combination of the remaining ones in columns. 

\subsection{Multiway correspondence analysis}
	PCA is solving dimension reduction problem for a matrix in $E \otimes F$, with $E = \R^{m}$ and $F = \R^{n}$ as natural choice in data science. In the geometric context, a cloud $\mathcal{A}$ of $m$ points in $\R^n$ is associated with a matrix $A$, where point $i$ is row $i$ of $A$ (points are in $F$). PCA of $A$ is building a new orthonormal basis in $F$, called principal axis, and computing the coordinates of the points in this new basis, called principal components. Classically PCA is realized with a SVD of the given matrix, i.e., $A=U\Sigma V^\top$. Then the principal axis are defined as the columns of $V$, and the array of coordinates is given by $Y=U\Sigma$. It can be developed \emph{mutatis mutandis} by selecting some inner products associated with SPD matrices in $E$ and/or $F$. Often in data analysis, those SPD matrices are diagonal, and the metrics are defined by weights. Indeed given $M$ in $\R^{m \times m}$ and $Q$ in $\R^{n \times n}$, we define the inner product $\langle A,B \rangle_{M \otimes Q}= \langle MAQ,MBQ\rangle$. Then, PCA of $A$ at rank $r$ with this inner product is finding a rank $r$ matrix $A_r \in \R^{m \times n}$ such that $\|A-A_r\|_{M \otimes Q}$ is minimal. In other words, we compute the SVD at rank $r$ of $X = \nu_M(A) = MAQ$. If $(Y,\Lambda, V)$ is the set of principal components $Y$ of $X$, eigenvalues $\Lambda = \Sigma^2$ and principal axis $V$, then, principal components and axis of PCA of $A$ with metrics so defined are $(M^{-1}Y, \Lambda, Q^{-1}V)$. In particular, as Figure \ref{figS:1L} shows, CA is a PCA on a contingency table with metrics associated with inverse of the marginals as weights~\cite{LMF82}.
	
	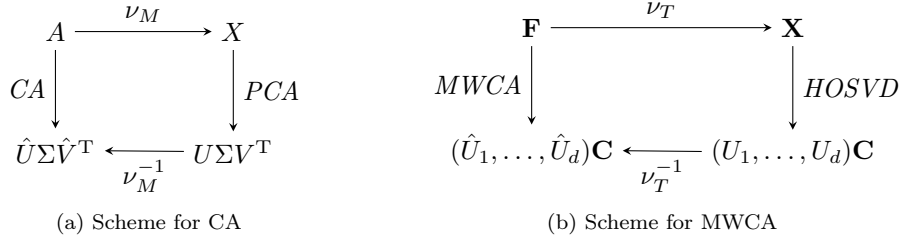
\begin{figure}[htb]
		\centering
		\subfloat[Scheme for CA]{
			\vspace{-5pt}
			\begin{tikzpicture}
			\matrix (m) [matrix of math nodes,row sep=3em,column sep=3em, ampersand replacement=\&]
			{A \&  X\\
				\hat{U}\Sigma \hat{V}^\top \& U\Sigma V^\top \\};
			\path[-stealth] (m-1-1) edge node [above] {$\nu_M$} (m-1-2);
			\path[-stealth] (m-1-2) edge node [right] {\textit{PCA}} (m-2-2);
			\path[-stealth] (m-2-2) edge node [below] {$\nu_M^{-1}$} (m-2-1);
			\path[-stealth] (m-1-1) edge node [left] {\textit{CA}} (m-2-1);
			\end{tikzpicture}\label{figS:1L}}
		\qquad\qquad
		\subfloat[Scheme for MWCA]{
			\begin{tikzpicture}
			\matrix (m) [matrix of math nodes,row sep=3em,column sep=3em,ampersand replacement=\&]
			{\ten{F} \&  \ten{X}\\
				(\hat{U}_1,\dots, \hat{U}_d)\vec{C} \& (U_1,\dots, U_d)\vec{C} \\};
			\path[-stealth] (m-1-1) edge node [above] {$\nu_T$} (m-1-2);
			\path[-stealth] (m-1-2) edge node [right] {\textit{HOSVD}} (m-2-2);
			\path[-stealth] (m-2-2) edge node [below] {$\nu_T^{-1}$} (m-2-1);
			\path[-stealth] (m-1-1) edge node [left] {\textit{MWCA}} (m-2-1);
			\end{tikzpicture}\label{figS:1R}}
		\caption{Correspondence analysis matrix and tensor approach}
	\end{figure}

	 This is extended naturally to a multiway contingency table $\textbf{T}$~\cite{Kroonenberg2008}, and formalized through HOSVD. In Figure \ref{figS:1R} we sketch this approach, which we call MWCA as it extends with HOSVD the method of CA. As preprocessing step we compute the relative frequency tensor $\ten{F}$ dividing each tensor $\ten{T}$ entry by the sum of all its entries. First step is to compute all marginals of $\mathbf{F}$ for all indices for all modes. This yields a vector of weights $\mathbf{w}^\mu = (w^{\mu}_{i_1}, \dots,w^{\mu}_{i_\mu})$ for mode $\mu$. In the second step the isometry $\nu_T$ is defined by the diagonal matrix $M_\mu$ whose diagonal elements are inverse of the weights. Third step is to perform HOSVD of $\ten{X} = \nu_T(\mathbf{F})$. Finally the HOSVD decomposition is transported back with the isometry inverse $\nu_T^{-1}$. 
	
\section{Multiway principal components analysis\label{sec:3}}
  Starting from the extension of principal component analysis to tensors with HOSVD, we associate a point cloud with each mode, and show the existence of an algebraic link between them. The aim of this section is interpreting from a geometrical point of view this relation. For the sake of simplicity in the result explicit verification, we first prove a link between the point clouds in the standard Euclidean $3$-order tensor space, and then we generalize to $d$-order tensors. This structure choice is kept throughout the document. We will focus especially on the geometric interpretation of these results. We naturally attach a point cloud to each matricization of $3$-order tensor. Each point cloud is the optimal projection of the mode matricization in low dimension space. Finally, we show how their coordinates are linked and we extend this result to general metric spaces of $d$-order tensors. 
  \par Let $\vec{X}\in E_1\otimes \dots\otimes E_d$ be a tensor with $\text{dim}(E_\mu) = n_\mu$ and let $(U_\mu)_\mu\in\ort(n_\mu \times r_\mu)$ be the rank $\vec{r} = (r_1,\dots, r_d)$ Tucker decomposition basis obtained from the HOSVD algorithm. The tensor $\vec{X}$ is expressed as
	\begin{equation}
	\label{eq3:0}
	  \vec{X} = \sum_{i_1,\dots, i_d = 1}^{r_1,\dots, r_d} C_{i_1\dots i_d}\vec{u}^1_{i_1}\otimes \dots\otimes \vec{u}^d_{i_d}
	\end{equation}
  with $\vec{C}$ the HOSVD core tensor and $\vec{u}^{\mu}_{i_{\mu}}$ the $i_{\mu}$-th column of $U_{\mu}$. The condition for Equation~\eqref{eq3:0} to be a decomposition is $r_\mu = \text{rank}(X^{(\mu)})$ for $\mu\in\sqb{1}{d}$. Let $\Sigma_\mu$ be the diagonal singular value matrix of the matricization of $\vec{X}$ with respect to mode $\mu$. For simplicity the $i_\mu$-th diagonal element of $\Sigma_\mu$ is denoted by $\sigma^{(\mu)}_{i_\mu}$ for every $i_\mu\in\sqb{1}{r_\mu}$ and for every $\mu\in\sqb{1}{d}$. The principal component of mode $\mu$ is $Y_\mu\in\R^{n_\mu\times r_\mu}$ defined as $Y_\mu = U_\mu\Sigma_\mu$ for each $\mu\in\sqb{1}{d}$.

\subsection{The $3$-order tensor case in the Euclidean space} 
  For the sake of simplicity and clarity, we assume $d$ equal to $3$. The following proposition states a relation linking the three sets of principal coordinates.
  
  \begin{proposition}
		\label{lmT:1} Let $\vec{X}$ be a tensor of $E_1\otimes E_2\otimes E_3$ with $\text{dim}(E_\mu) = n_\mu$ and let $\vec{C}$ be its HOSVD core at multi-linear rank $\vec{r}$. Let $Y_\mu$ be the principal components of mode $\mu$. If $r_\mu \le \text{rank}({A}^{(\mu)})$ for $\mu\in\{1,2,3\}$, then 

		\begin{equation*}
		\begin{split}
		Y_1 &= {X}^{(1)}(Y_3\otimes_K Y_2)({B}^{(1)})^\top\\
		Y_2 &= {X}^{(2)}(Y_3\otimes_K Y_1)({B}^{(2)})^\top\\\
		Y_3 &= {X}^{(3)}(Y_2\otimes_K Y_1)({B}^{(3)})^\top
		\end{split}
		\end{equation*}

		with $\vec{B} = (\Sigma_1^{-1},\Sigma_2^{-1},\Sigma_3^{-1})\vec{C}$. 
	\end{proposition}
	
	\begin{proof}[\textbf{\upshape Proof:}]We start the proof for the first mode principal components. Let $\vec{X}$ be expressed in the HOSVD basis as in Equation \eqref{eq3:0}, i.e.,

		\begin{equation*}
		\vec{X}=\sum_{i,j,k = 1}^{r_1,r_2,r_3}C_{ijk}\vec{u}^1_{i}\otimes \vec{u}^2_{j}\otimes \vec{u}^3_{k}.
		\end{equation*}

		Then the matricization of $\vec{X}$ with respect to mode $1$ in the Tucker basis is expressed with the Kronecker product as
		\begin{equation}
		\label{eq3:1}
		X^{(1)} = \sum_{i=1}^{r_1} \vec{u}^1_{i}\otimes \biggl(\sum_{j,k=1}^{r_2,r_3}C_{ijk}\vec{u}^{3}_{k}\otimes_K\vec{u}^{2}_{j}\biggr).
		\end{equation}
		The PCA of ${X}^{(1)}$ is
		\begin{equation}
		\label{eq3:2}
		{X}^{(1)} = Y_1V_1^\top = \sum_{i=1}^{r_1}\sigma^{(1)}_{i} \vec{u}^{1}_{i}\otimes \vec{v}^{1}_{i}
		\end{equation}
		with $\vec{v}_i^1$ the $i$-th column of $V_1\in\ort(n_2n_3\times r_1)$ and $\sigma^{(1)}_i\vec{u}^{1}_i$ $i$-th column of $Y_1\in\R^{n_1\times r_1}$. 
		By comparing Equations \eqref{eq3:1} and \eqref{eq3:2} for a fixed index $i$, we get
		\begin{equation*}
		\sigma^{(1)}_i \vec{u}^{1}_i\otimes \vec{v}^{1}_i = \vec{u}^1_{i}\otimes\biggl( \sum_{j,k=1}^{r_2,r_3}C_{ijk}\vec{u}^{3}_{k}\otimes_K\vec{u}^{2}_{j}\biggr).
		\end{equation*}

		Remarking that $\Sigma_1$ is invertible, we identify $\vec{v}^1_i$ with a linear combination of the Kronecker product of $\vec{u}^2_j$ and $\vec{u}^3_k$ scaled by $\sigma^{(1)}_i$ as
		\begin{equation}
		\label{eq3:3}
		\vec{v}^{1}_i = \frac{1}{\sigma^{(1)}_i}\sum_{j,k=1}^{r_2,r_3}C_{ijk}\vec{u}^{3}_k\otimes_K \vec{u}^2_{j}.
		\end{equation}
		Notice that the $j$-th and $k$-th column of $Y_2=U_2\Sigma_2$ and $Y_3=U_3\Sigma_3$ are $\sigma^{(2)}_j\vec{u}^{2}_j$ and $\sigma^{(3)}_k\vec{u}^{3}_k$ respectively.
		So introducing in Equation \eqref{eq3:3} the singular values $\sigma^{(2)}_{j}$ and $\sigma^{(3)}_{k}$, we express the $i$-th column of $V_1$ as a linear combination of the Kronecker product of the $j$-th and $k$-th column of $Y_2$ and $Y_3$, i.e.
		
		\begin{equation*}
		\label{eq3:4}
		\begin{split}
  		\vec{v}^{1}_i &=
  		\frac{1}{\sigma^{(1)}_i}\sum_{j,k=1}^{r_2,r_3}C_{ijk}\frac{\sigma^{(2)}_{j}\sigma^{(3)}_{k}}{\sigma^{(2)}_{j}\sigma^{(3)}_{k}}\,\vec{u}^{3}_k\otimes_K \vec{u}^2_{j}\\
  		&= \sum_{j,k=1}^{r_2,r_3}\frac{C_{ijk}}{\sigma^{(1)}_i\sigma^{(2)}_{j}\sigma^{(3)}_{k}}\,\vec{y}^3_k\otimes_K\vec{y}^2_j\\
  		&=\sum_{j,k=1}^{r_2,r_3}B_{ijk}\,\vec{y}^3_k\otimes_K\vec{y}^2_j
		\end{split}
		\end{equation*}
		with $\vec{B} = (\Sigma_1^{-1},\Sigma_2^{-1},\Sigma_3^{-1})\vec{C}$, $\vec{y}^2_j$ and $\vec{y}^3_k$ the $j$-th and $k$-th column of $Y_2$ and $Y_3$ respectively .\\
		Remark that $Y_3\otimes Y_2$ is a matrix of $n_2n_3$ rows and $r_2r_3$ columns whose $\ell$-th column is $\;\vec{y}^3_k\otimes_K\vec{y}^2_j$ with $\ell = \overline{jk}$ for every $j\in\sqb{1}{r_2}$, $k\in\sqb{1}{r_3}$ and $\ell\in \sqb{1}{r_2r_3}$, as defined in Equation~\eqref{eq2:1}. The tensor $\vec{B}$ matricized with respect to mode $1$ is a matrix of $r_1$ rows and $r_2r_3$ columns, whose $(i,\overline{jk})$-th element is $b_{ijk}$ for all $j\in\sqb{ 1}{r_2}$, $k\in\sqb{1}{r_3}$. So the sum in the right-hand side of Equation \eqref{eq3:4} can be expressed as the matrix-product between $Y_3\otimes_K Y_2$ and tensor $\vec{B}$ matricized with respect to mode $1$ as
		\begin{equation}
	  \label{eq3:5}
	  V_1 =(Y_3\otimes_K Y_2)({B}^{(1)})^\top.
		\end{equation}
		
		Multiplying Equation \eqref{eq3:2} on the right by $V_1$ yields $Y_1 = {X}^{(1)}V_1$. Therefore 
		multiplying Equation \eqref{eq3:5} by the matricization of $\vec{X}$ with respect to mode $1$, the principal component $Y_1$ is expressed as linear combination of the Kronecker product of principal components $Y_2$ and $Y_3$, i.e.
		\begin{equation*}
		Y_1 = {X}^{(1)}V_1 = {X}^{(1)}(Y_3\otimes_K Y_2)({B}^{(1)})^\top.
		\end{equation*}
		The other relations follow straightforwardly from this one, permuting the indices coherently.
	\end{proof}
	From an algebraic view this first proposition shows that the principal components of each mode can be expressed as a linear combination of the principal components of the two other modes. However as pointed out in the preliminary section, principal components can be seen from different viewpoints. 
	\par From a geometric viewpoint a point cloud $\mathcal{X}_\mu$ is attached to each mode $\mu$ matricization of tensor $\vec{X}$ for $\mu\in\{1,2,3\}$. Indeed the $i_\mu$-th row of ${X}^{(\mu)}$ represents the coordinates of the $i_\mu$-th element of mode $\mu$ point cloud living in the space $\R^{n_{\ne \mu}}$ where $n_{\ne\mu} = \frac{n_1n_2n_3}{n_{\mu}}$. Given a multi-linear rank $\vec{r}$, we reformulate the problem of the Tucker approximation as a problem of dimensional reduction. Indeed we look for the subspace of $\R^{n_{\ne\mu}}$ of dimension $r_{\mu}$ which minimizes in norm the projection of point cloud $\mathcal{X}_{\mu}$ on it. This problem is solved with the HOSVD algorithm, which provides three orthogonal basis $(U_1, U_2, U_3)$ of the corresponding subspaces. Therefore the $i_\mu$-th row of $Y_{\mu}=U_{\mu}\Sigma_{\mu}$ represents the coordinates of the $i_\mu$-th element of $\mathcal{X}_\mu$ projected into the subspace of $\R^{n_{\ne\mu}}$ of dimension ${r}_{i_\mu}$ for $\mu\in\{1,2,3\}$. The Proposition \ref{lmT:1} result is interpreted geometrically as each point cloud living in the linear subspace built from the Kronecker product of the other two.

\subsection{Generalization to $d$-order tensors}
  Now, we generalize Proposition \ref{lmT:1} to the $d$-case as follows.
	
	\begin{proposition}
		\label{lmT:3} Let $\vec{X}$ be a tensor of $E_1\otimes \dots\otimes E_d$ with $\text{dim}(E_\mu) = n_\mu$ and let $\vec{C}$ be its HOSVD core at multi-linear rank $\vec{r}$. Let $Y_\mu$ be the principal components of mode $\mu$. If ${r}_\mu \le \text{rank}({A}^{(\mu)})$ for $\mu\in\sqb{1}{d}$, then 
	  \begin{equation*}
		Y_{\mu} = {X}^{(\mu)}(Y_{d}\otimes_K\dots\otimes_K Y_{\mu +1 }\otimes_K Y_{\mu - 1}\otimes_K\dots\otimes_K Y_1)({B}^{(\mu)})^\top
	  \end{equation*}
		with $\vec{B} = (\Sigma_1^{-1},\dots,\Sigma_d^{-1})\vec{C}$ for every $\mu\in\sqb{1}{d}$. 
	\end{proposition}
	
	\begin{proof}[\textbf{\upshape Proof:}] The proof is very similar to that of Proposition \ref{lmT:1}, so we give only the main steps. Let start by the first mode principal components. 
	The $1$-mode matricization of $\vec{X}$ in the Tucker basis is expressed with the Kronecker product, getting
  	\begin{equation}
  	\label{eq3:14}
  	X^{(1)} = \sum_{i_1=1}^{r_1} \vec{u}^1_{i_1}\otimes \biggl(\sum_{i_2, \dots, i_d=1}^{r_2,\dots, r_d}C_{i_1\dots i_d}\vec{u}^{d}_{i_d}\otimes_K\dots\otimes_K\vec{u}^{2}_{i_2}\biggr).
  	\end{equation}
		The PCA of ${X}^{(1)}$ leads to
  	\begin{equation}
  	\label{eq3:15}
  	{X}^{(1)} = Y_1V_1^\top = \sum_{i_1=1}^{r_1}\sigma^{(1)}_{i_1} \vec{u}^{1}_{i_1}\otimes \vec{v}^{1}_{i_1}
  	\end{equation}
		with $\vec{v}_{i_1}^1$ the $i_1$-th unitary column of $V_1\in\ort(n_{\ne 1}\times r_1)$ where $n_{\ne 1} = \prod_{\mu=2}^d n_{\mu}$ and $\sigma^{(1)}_{i_1}\vec{u}^{1}_{i_1}$ the $i_1$-th column of $Y_1\in\R^{n_1\times r_1}$. 
		Comparing Equations \eqref{eq3:14} and \eqref{eq3:15} for a fixed index $i_1$, we identify $\vec{v}^1_{i_1}$ with a linear combination of the Kronecker product of $\vec{u}^{\mu}_{i_\mu}$ for $\mu\in\sqb{2}{d}$ scaled by $\sigma^{(1)}_{i_1}$ as
  	\begin{equation}
  	\label{eq3:16}
  	\vec{v}^{1}_{i_1} = \frac{1}{\sigma^{(1)}_{i_1}}\sum_{i_2, \dots, i_d = 1}^{r_2,\dots,r_d}C_{i_1\dots i_d}\vec{u}^{d}_{i_d}\otimes_K\dots\otimes_K \vec{u}^2_{i_2}.
  	\end{equation}
		Introducing in Equation \eqref{eq3:16} the singular values $\sigma^{(\mu)}_{i_\mu}$, we express the $i_1$-th column of $V_1$ as a linear combination of the Kronecker product of the $i_\mu$-th column of $Y_\mu$ for every $i_\mu\in\sqb{2}{d}$ as
		\begin{equation}
		\label{eq3:17}
      \begin{split}
      \vec{v}^{1}_{i_1} &=\sum_{i_2, \dots, i_d = 1}^{r_2,\dots,r_d}\frac{C_{i_1\dots i_d}}{\sigma^{(1)}_{i_1}\sigma^{(2)}_{i_2}\dots \sigma^{(d)}_{i_d}} \,\sigma^{(2)}_{i_2}\dots\sigma^{(d)}_{i_d}\,\vec{u}^{d}_{i_d}\otimes_K\dots\otimes_K \vec{u}^2_{i_2}\\
      &=\sum_{i_2, \dots,i_d=1}^{r_2,\dots,r_d}B_{i_1\dots i_d}\,\vec{y}^d_{i_d}\otimes_K\dots\otimes_K\vec{y}^2_{i_2}
      \end{split}
		\end{equation}
		with $\vec{B} = (\Sigma_1^{-1},\dots,\Sigma_d^{-1})\vec{C}$ and $\vec{y}^{\mu}_{i_\mu}$ the $i_\mu$-th column of $Y_\mu$ for $\mu\in\sqb{2}{d}$. Thanks to the correspondence between Kronecker product and matricization, the right-hand-side of Equation \eqref{eq3:17} is expressed as the matrix-product between $Y_d\otimes_K\dots\otimes_K Y_2$ and tensor $\vec{B}$ matricized with respect to mode $1$ and transposed as
  	\begin{equation}
  	\label{eq3:18}
  	V_1 =(Y_d\otimes_K\dots\otimes_K Y_2)({B}^{(1)})^\top.
  	\end{equation}
		Multiplying Equation \eqref{eq3:15} on the right by $V_1$ yields $Y_1 = {X}^{(1)}V_1$ and replacing $V_1$ by its expression of Equation \eqref{eq3:18}, it finally follows
		\begin{equation*}
		 Y_1 = {X}^{(1)}V_1 = {X}^{(1)}(Y_d\otimes_K\dots\otimes_K Y_2)({B}^{(1)})^\top.
		\end{equation*}
		The other relations follow straightforwardly from this one, permuting the indices coherently.
	\end{proof}

\subsection{Extension to generic metric space for $d$-order tensors\label{sec:4}}
  As discussed in Section~\ref{sec:2}, the minimization problem faced with the HOSVD algorithm is expressed by the Frobenious norm, induced by an inner product. The standard inner product is defined by the identity matrix. However, whatever SPD matrix induces an inner product and the associated metric norm on a vector space, which is therefore isomorphic to the standard Euclidean space. We emphasize in this section the role of the metric on the relationships between point clouds, using this isomorphic relationship between Euclidean spaces with different inner products.
  
  \par Let $\vec{F}\in\mathcal{S}_M$ where $\mathcal{S}_M$ is the tensor space $E_1\otimes \dots\otimes E_d$ with $\text{dim}(E_\mu)=n_\mu$ endowed with the inner product induced by SPD matrices $M_\mu$ of size $n_\mu$. Let $\mathcal{S}$ be the Euclidean tensor space $E_1\otimes \dots\otimes E_d$ endowed with the standard inner product. As already mentioned, we can move back to $\mathcal{S}$, thanks to the isometry define as $\nu: \mathcal{S}_M\rightarrow \mathcal{S}$, such that $\vec{X}=\nu(\vec{F}) =(M_1, \dots, M_d)\vec{F}$. Let now $\vec{X}\in\S$ be the HOSVD approximation of $\nu(\vec{F})$ at multi-linear rank $\vec{r}$ and let $(U_\mu)_\mu \in\ort(n_\mu \times r_\mu)$ be the associated basis. $\Sigma_\mu$ is the singular value matrix of $X^{(\mu)}$ and define $Y_\mu = U_\mu\Sigma_\mu$ the principal components of mode $\mu$ for tensor $\vec{X}$ for every $\mu\in\sqb{1}{d}$. Let $W_\mu = M^{-1}_\mu Y_\mu$ be the principal components of $\vec{F}$ in the tensor space $\mathcal{S}_M$. In the following proposition, we link the sets of principal components in the metric space $\mathcal{S}_M$. As previously, for the sake of clarity the result is first proved for $d=3$ and afterwards geralized to whatever order $d$. 
	
	\begin{proposition}
		\label{lmT:2}
		Let $\vec{F}$ be a tensor in $\mathcal{S}_M$ and let $\vec{X}$ be its image through the isometry $\nu$ in the standard tensor space $\mathcal{S}$. Let $\vec{C}$ be the HOSVD core tensor of $\vec{X}$ at multi-linear rank $\vec{r}$ such that ${r}_\mu \le \text{rank}(X^{(\mu)})$ for $\mu\in\{1,2,3\}$. Then for $W_\mu$ the principal component of mode $\mu$ of $\vec{F}$ in the metric tensor space $\mathcal{S}_M$ it holds
		\begin{equation*}
			\begin{split}
			W_1 &= F^{(1)}(M_3^2W_3\otimes_K M_2^2W_2)(B^{(1)})^\top\\
			W_2 &= F^{(2)}(M^2_3W_3\otimes_K M^2_1W_1)({B}^{(2)})^\top\\
			W_3 &= F^{(3)}(M^2_2W_2\otimes_K M^2_1W_1)({B}^{(3)})^\top
			\end{split}
		\end{equation*}
		with $\vec{B} = (\Sigma_1^{-1},\Sigma_2^{-1},\Sigma_3^{-1})\vec{C}$.
	\end{proposition}
	
	\begin{proof}[\textbf{\upshape Proof:}] This result comes straightforwardly from Proposition~\ref{lmT:1} proof by introducing the metrics matrices. For completeness, we illustrate the proof focusing on the link for the principal components of the first mode. Let $(U_\mu)_{\mu = 1,2,3}$ be the HOSVD basis of $\vec{X}$ at multi-linear rank $\vec{r}$ with $U_\mu \in \ort(n_\mu \times r_\mu)$. Let $Y_\mu = U_\mu\Sigma_\mu$ be the principal components of mode $\mu$ for tensor $\vec{X}$, where $\Sigma_\mu$ is the singular values matrix of $X^{(\mu)}$. Proposition \ref{lmT:1} yields 
		\begin{equation}
		\label{eq3:6}
		Y_1 = {X}^{(1)}(Y_3\otimes_K Y_2)(B^{(1)})^\top
		\end{equation}
		with $\vec{B} = (\Sigma_1^{-1},\Sigma_2^{-1},\Sigma_3^{-1})\vec{C}$. Notice that $\vec{X} = \nu(\vec{F}) = (M_1,M_2,M_3)\vec{F}$. So thanks to Equation~\eqref{rk:1}, we express $X^{(1)}$ in function of $F^{(1)}$ as 
		\begin{equation*}
			X^{(1)} = M_1F^{(1)}(M^\top_3\otimes_K M^\top_2)
		\end{equation*}
		and replacing it into \eqref{eq3:6}, it gets 
		\begin{equation}
		\label{eq3:7}
		\begin{split}
		Y_1 = M_1F^{(1)}(M_3\otimes_K M_2)(Y_3\otimes_K Y_2)(B^{(1)})^\top
		\end{split}
		\end{equation}
		since $M_\mu$ are SPD matrices. Remarking that $Y_\mu = M_\mu W_\mu$ from the $W_\mu$ definition, substituting it in the Equation \eqref{eq3:7} we obtain
		\begin{equation}
		\label{eq3:8}
		M_1 W_1 = M_1F^{(1)}( M_3\otimes_K M_2) (M_3W_3 \otimes_K M_2 W_2 )(B^{(1)})^\top.
		\end{equation}
		Since $M_1$ is SPD and consequently invertible, from the previous equation it follows the thesis. 
		\par The other relations follow straightforwardly from this proof, permuting the indices coherently.
	\end{proof}
	This result is easily generalized to $d$-order tensors as follows.
  	
  	\begin{proposition}
		\label{lmT:4}
		Let $\vec{F}$ be a tensor in $\mathcal{S}_M$ and let $\vec{X}$ be its image through the isometry $\nu$ in the standard tensor space $\mathcal{S}$. Let $\vec{C}$ be the HOSVD core tensor of $\vec{X}$ at multi-linear rank $\vec{r}$ such that ${r}_\mu \le \text{rank}(X^{(\mu)})$ for for $\mu\in\sqb{1}{d}$. Then for $W_\mu$ the principal component of mode $\mu$ of $\vec{F}$ in the metric tensor space $\mathcal{S}_M$ it holds
		\begin{equation*}
				W_{\mu} = F^{(1)}(M_d^2W_d\otimes_K\dots\otimes_K M_{\mu +1}^2W_{\mu +1}\otimes_KM_{\mu-1}^2W_{\mu-1}\otimes_K\dots \otimes_KM_{1}^2W_{1})(B^{(\mu)})^\top
		\end{equation*}
		with $\vec{B} = (\Sigma_1^{-1},\dots,\Sigma_d^{-1})\vec{C}$ for every $\mu\in\sqb{1}{d}$.
	\end{proposition}
	
	\begin{proof}[\textbf{\upshape Proof:}]
	The proof is a direct consequence of Proposition \ref{lmT:3} and Proposition \ref{lmT:2}, so the details are omitted here.
	\end{proof}

\section{Geometric view for multiway correspondence analysis\label{sec:5}}
  In this last section, we transport the previous results in the correspondence analysis framework. We firstly clarify the Euclidean space and its metric where we set our problem. Then we make explicit the point cloud relation in this particular context. As final outcome we are able to prove the correspondence between the point clouds attached to each mode.
  
  \par In accordance with the correspondence analysis framework, we consider a $d$-way contingency table $\vec{T}\in\N^{n_1\times \dots\times n_d}$. The first step for performing CA is scaling $\vec{T}$ by the sum of all its components setting a new frequency tensor 
	\[\vec{F} = \frac{1}{\sum_{i_1, \dots, i_d = 1}^{n_1,\dots,n_d}{T}_{i_1\dots i_d}}\vec{T}.\] 
	We first clarify the tensor space we will work with. Let $\vec{f}^{\mu}$ be the marginal of mode $\mu$, i.e., the vector whose components are the sums of the slices in mode $\mu$ for all $\mu\in\sqb{1}{d}$. For example the $i$-th element of $\vec{f}^1$ is
	\[{f}^{1}_{i_1} =\sum_{i_2, \dots , i_d =1}^{n_2,\dots, n_d}{F}_{i_1\dots i_d}\qquad\text{for all}\qquad i_1\in\sqb{1}{n_1}.\] 
	We assume that $\vec{f}^{\mu}$ has no zero component and by construction ${f}^{\mu}_{i_\mu} > 0 $ for every $\mu\in\sqb{1}{d}$. Let define $D_\mu = diag(\sqrt{\vec{f}^{\mu}})\in\R^{n_\mu\times n_\mu}$ for each $\mu \in\sqb{1}{d}$ and let assume that $\vec{F}$ belongs to $\R^{n_1\times \dots\times n_d}$ endowed by the metric induced by the matrices $(D^{-1}_1,\dots,D^{-1}_d)$, since $D_\mu^{-1}$ is SPD for every $\mu\in\sqb{1}{d}$. We denote by $\S_M$ this metric space and by $\S$ the tensor space $\R^{n_1\times \dots\times n_d}$ endowed with the standard inner product. Under this assumption, let $\nu$ the isometry between the spaces $\S_M$ and $\S$ and let $\vec{X}=\nu(\vec{F}) = (D^{-1}_1,\dots,D^{-1}_d)\vec{F}$. The general element of tensor $\vec{X}$ is written 
	\begin{equation*}
	  {X}_{i_1\dots i_d} = \frac{{F}_{i_1\dots i_d}}{\sqrt{{f}^1_{i_1}\dots{f}^d_{i_d}}}.
	\end{equation*}
	Performing the HOSVD over tensor $\vec{X}$ at multi-linear rank $\vec{r}$ leads to a new orthogonal basis $(U_\mu)_{\mu=1,\dots,d}$, to a core tensor $\vec{C}$ and to the principal components $Y_\mu = U_\mu \Sigma_\mu$ for every $\mu\in\{1,\dots, d\}$ in the standard tensor space. Focusing on the principal components $W_\mu = D_\mu Y_\mu$ of tensor $\vec{F}$ in $\S_M$, Proposition~\ref{lmT:4} entails
	\begin{equation*}
  		W_{\mu} = F^{(1)}(M_d^2W_d\otimes_K\dots\otimes_K M_{\mu +1}^2W_{\mu +1}\otimes_KM_{\mu-1}^2W_{\mu-1}\otimes_K\dots \otimes_KM_{1}^2W_{1})(B^{(\mu)})^\top
	\end{equation*}
	where $B^{(\mu)}$ is the matricization of $\vec{B} = (\Sigma^{-1}_1, \dots, \Sigma^{-1}_d)\vec{C}$ for $\mu\in\sqb{1}{d}$. 
	Let $Z_{\mu} = D^{-2}_\mu W_{\mu}$ be the \emph{principal components scaled} by the singular value inverse for $\mu\in\sqb{1}{d}$ in tensor space $\S_M$. Henceforth, we denote by $\vec{z}^{\mu}_{i_\mu}$ the $i_\mu$-th row of $Z_\mu$. 
	Now we prove that each component of vector $\vec{z}^{\mu}_{i_\mu}$ can be expressed as a scaling factor times the barycenter of the linear combinations of the other two scaled principal component rows. We assume $d$ equal to $3$ to facilitate the comprehension of the following proof.
	
	\begin{proposition}
  	\label{lmCA:1}
  	Let $\vec{F}$ be a tensor in the tensor space $\mathcal{S}_M$ endowed with the norm induced by the inner product matrices $D_\mu = \text{diag}(\sqrt{\vec{f}^{\mu}})$ with $\vec{f}^{\mu}$ the $\mu$ mode marginal of $\vec{F}$ for $\mu\in\{1,2,3\}$. Let $Z_\mu\in\R^{n_\mu\times r_\mu}$ be the scaled principal components for tensor $\vec{F}$ of mode $\mu$ in $\S_M$. If $r_\mu= \text{rank}(F^{(\mu)})$ for every $\mu\in\{1,2,3\}$, then
  	\begin{equation*}
	  \begin{split}
	    ({z}^1_i)_{\ell} &= \frac{1}{\sigma^1_{\ell}}\sum_{j,k=1}^{n_2,n_3}\sum_{{m,p}=1}^{r_2,r_3} \dfrac{{F}_{i j k}}{{f}^1_{i}}({z}^3_k)_{p}({z}^2_j)_{m}({B_1})_{\ell mp}\\
	    ({z}^2_j)_{m} &= \frac{1}{\sigma^2_{m}}\sum_{i,k=1}^{n_1,n_3}\sum_{{\ell,p}=1}^{r_1,r_3} \dfrac{{F}_{i j k}}{{f}^2_{j}}({z}^3_k)_{p}({z}^1_i)_{\ell}(B_2)_{\ell m p}\\
	    ({z}^3_k)_{p} &= \frac{1}{\sigma^3_{p}}\sum_{i,j=1}^{n_1,n_2}\sum_{{\ell,m}=1}^{r_1,r_2} \dfrac{{F}_{i j k}}{{f}^3_{k}}({z}^1_i)_{\ell}({z}^2_j)_{m}(B_3)_{\ell mp}
	  \end{split}
  	\end{equation*}
  	where $B_\mu = \Sigma_\mu B^{(\mu)}$ from Proposition \ref{lmT:2}.
  \end{proposition}
  	 
  \begin{proof}[\textbf{\upshape Proof:}] 
  We describe the proof for the $i$-th row of $Z_1$ with $i\in\sqb{1}{n_1}$. From Proposition \ref{lmT:2}, under the CA metric choice, it follows that the principal components of $\vec{F}$ in the tensor space $\S_M$ satisfy the relation
  \begin{equation*}
    W_1 = F^{(1)}(D_3^{-2}W_3\otimes_K D^{-2}_2W_2)(B^{(1)})^\top = F^{(1)}(Z_3\otimes_K Z_2)(B^{(1)})^\top. 
  \end{equation*}
  Multiplying on the left this last equation by $D^{-2}_1$, we obtain
  \begin{equation*}
    Z_1 = D^{-2}_1W_1 = D^{-2}_1F^{(1)}(Z_3\otimes_K Z_2)(B^{(1)})^\top
    \end{equation*}
  and since $B^{(1)}=\Sigma^{-1}_1{B}^{(1)}_1$, it gets
  \begin{equation}
  \label{eq3:10}
    Z_1 = D^{-2}_1F^{(1)}(Z_3\otimes_K Z_2)(\Sigma_1^{-1}B_1^{(1)})^\top.
  \end{equation}
  Making explicit the $\ell$-th component of $\vec{z}^1_i$, the $i$-th row of $Z_1$, from Equation \eqref{eq3:10}, we have
  \begin{equation*}
    \begin{split}
    ({z}^1_i)_{\ell} =(Z_1)_{i\ell}&=
    \sum_{j,k=1}^{n_2,n_3}\sum_{{m,p}=1}^{r_2,r_3} (D^{-2}_1F^{(1)})_{i j k}(Z_3\otimes_K Z_2)_{\overline{jk}\,\overline{mp}}(\Sigma_1^{-1}B_1^{(1)})_{\ell\,\overline{mp}}\\
    &=
    \sum_{j,k=1}^{n_2,n_3}\sum_{{m,p}=1}^{r_2,r_3} \dfrac{{F}_{i j k}}{{f}^1_{i}}(Z_3)_{kp}(Z_2)_{jm}\dfrac{(B_1)_{\ell m p}}{\sigma^1_{\ell}}\\
    &= \frac{1}{\sigma^1_{\ell}}\sum_{j,k=1}^{n_2,n_3}\sum_{{m,p}=1}^{r_2,r_3} \dfrac{{F}_{i j k}}{{f}^1_{i}}({z}^3_k)_{p}({z}^2_j)_{m}(B_1)_{\ell m p}
    \end{split}
  \end{equation*}
  by the definition of $\vec{z}^2_j$ and $\vec{z}^3_k$ for every $i\in\sqb{1}{n_1}$, $j\in\sqb{1}{n_2}$, $k\in\sqb{1}{n_3}$ and $\ell\in\sqb{1}{r_1}$. This final equation can be read as a mutual barycenter relation scaled by the inverse of the corresponding singular value. Indeed there is a list of weights terms which sum to 1, i.e., $\sum_{{j,k}=1}^{n_2, n_3}\dfrac{{F}_{i j k}}{{f}^1_{i}} = \frac{{f}^1_i}{{f}^1_i}=1$ times a linear combination expressed through $\vec{B}_1$ of $\vec{z}^2_j$ and $\vec{z}^3_k$. Moving back to the geometric perspective, the Proposition~\ref{lmCA:1} states that the scaled coordinates of a point cloud correspond to the barycenter of the other two point cloud scaled coordinates.
  \par The two remaining barycentric relations follow from this proof, permuting coherently the indices.
  \end{proof}
  Correspondence in CA refers to the correspondence of point cloud coordinates through the scaled barycentric relation. We proved that this well known relation in matrix framework is holding also in the tensor one, through HOSVD. Therefore we propose to refer to it as \emph{correspondence analysis from HOSVD}.
  
  \par This final proposition is extended and verified straightforwardly for $d$-order tensors.
  \begin{proposition}
  \label{lmCA:2}
    Let $\vec{F}$ be a tensor in the tensor space $\mathcal{S}_M$ endowed with the norm induced by the inner product matrices $D_\mu = \text{diag}(\sqrt{\vec{f}^{\mu}})$ with $\vec{f}^{\mu}$ the $\mu$ mode marginal of $\vec{F}$ for every $\mu\in\sqb{1}{d}$. Let $Z_\mu\in\R^{n_\mu\times r_\mu}$ be the scaled principal components for tensor $\vec{F}$ of mode $\mu$ in $\S_M$. If $r_\mu= \text{rank}(F^{(\mu)})$ for $\mu\in\sqb{1}{d}$, then 
    \begin{equation*}
  		({z}^{\mu}_{i_\mu})_{\ell_\mu} = \frac{1}{\sigma^{\mu}_{\ell_\mu}}\sum_{\substack{\eta =1\\\eta \ne \mu}}^{d}\sum_{i_\eta=1}^{n_\eta}\sum_{\ell_\eta = 1}^{r_\eta} \dfrac{{F}_{i_1\dots i_d}}{{f}^{\mu}_{i_\mu}}({z}^d_{i_{d}})_{\ell_d}\cdots({z}^{\mu+1}_{i_{\mu+1}})_{\ell_{\mu+1}}({z}^{\mu-1}_{i_{\mu-1}})_{\ell_{\mu-1}}\cdots({z}^2_{i_2})_{\ell_2}(B_\mu)_{\ell_1\dots\ell_d}
  	\end{equation*}
    where $B_\mu = \Sigma_\mu B^{(\mu)}$ from Proposition \ref{lmT:4}.
  \end{proposition}
  \begin{proof}[\textbf{\upshape Proof:}] 
  From Proposition \ref{lmT:4}, under the CA metric choice, the $1$st mode principal components of $\vec{F}$ in $\S_M$ are expressed as 
    \begin{equation*}
    W_1 = F^{(1)}(Z_d\otimes_K\dots\otimes_K Z_2)(B^{(1)})^\top. 
    \end{equation*}
  If the previous equation is multiplied by equation by $D^{-2}_1$ and it has ${B}^{(1)}$ replaced by the equivalent $\Sigma^{-1}_1{B}^{(1)}_1$, it gets
    \begin{equation}
    \label{eq4:1}
    Z_1 = D^{-2}_1F^{(1)}(Z_d\otimes_K\dots\otimes_K Z_2)(\Sigma_1^{-1}B_1^{(1)})^\top.
    \end{equation}
  Making explicit the $\ell_1$-th component of $\vec{z}^1_{i_1}$, the $i_1$-th row of $Z_1$, from Equation \eqref{eq4:1}, it follows the thesis, i.e.
  \begin{equation*}
    \begin{split}
    ({z}^1_{i_1})_{\ell_1} = \frac{1}{\sigma^1_{\ell_1}}\sum_{i_2,\dots, i_d=1}^{n_2,\dots,n_d}\sum_{\ell_2,\dots, \ell_d=1}^{r_2,\dots,r_d} \dfrac{{F}_{i_1\dots i_d}}{{f}^1_{i_1}}({z}^d_{i_d})_{\ell_d}\cdots({z}^2_{i_2})_{\ell_2}(B_1)_{\ell_1\dots\ell_d}.
    \end{split}
  \end{equation*}
  The proof for the other modes follows directly from this one permuting coherently the indices.
  \end{proof}

\section{Applications\label{sec:6}}
	In this section we compare the MWCA based on the point cloud relation proved in Section \ref{sec:5} with the classical CA performed on the same data reorganized as a matrix. Indeed, if we assume $\ten{F}\in\R^{n_1\times \dots\times n_d}$ to store the data relative frequencies as a tensor, then their matrix representation is given by $A_k$, the mode $k$-matricization of $\ten{F}$, once mode $k$ has been selected. The MWCA is performed on $\ten{F}$, while CA is performed on $A_k$. Let denote by $\tilde{{F}}^{(k)}$ the $k$-matricization of $\ten{F}$ after applying the isometry, similarly $\tilde{A}_k$ is the outcome of isometry in the CA case. Since the two approaches are different, the isometries transporting $\ten{F}$ and $A_k$ differ and so do $\tilde{A}_k$ and $\tilde{{F}}^{(k)}$. To estimate this discrepancy between the two objects we define the relative error $e(\tilde{{F}}^{(k)}, \tilde{A}_k)$
	\begin{equation}\label{eqn:errorCA-MCWA}
	     e(\tilde{{F}}^{(k)}, \tilde{A}_k) = \frac{||\tilde{{F}}^{(k)} - \tilde{A}_k||}{||\tilde{{F}}^{(k)}||}.
	\end{equation}
	The SVD and HOSVD, over which CA and MWCA relay respectively, provide an orthogonal basis for the matrix or the decomposed tensor, which is unique up to an orthogonal rotation. As proposed in~\cite{BroSVD} and~\cite{BroSignIndeter}, we orient these new basis selecting as leading direction the one where the majority of the data points out. To perform CA and MWCA we used \texttt{python 3.6.9} and the library \texttt{TensorLy 0.6.0}, see~\cite{tensorly}.
	\par We first analyse with these two techniques the data reported in~\cite{greenacre1993Data}. This example is important since it shows that the multiway method results are coherent with those stated with correspondence analysis. Then the multiway barycentric relation is used to interpret an original data-set from the ecological domain~\cite{Malabardata}.
	
	\begin{example}
	  \begin{table}[H]
	    \centering
	    \begin{tabular}{ll ccccc ccc ccccc}
	    \toprule
	     && \multicolumn{5}{c}{\textbf{Males}} &&&& \multicolumn{5}{c}{\textbf{Females}} \\
	    \cmidrule(lr){3-7} \cmidrule(lr){10-15}
	    \thead{\textit{Age}\\ \textit{group}}   && \thead{\textbf{Very} \\ \textbf{good}}  & \thead{\textbf{Good}} & \thead{\textbf{Regular}} & \thead{\textbf{Bad}} & \thead{\textbf{Very}\\ \textbf{bad}} &&&& \thead{\textbf{Very} \\ \textbf{good}}  & \thead{\textbf{Good}} & \thead{\textbf{Regular}} & \thead{\textbf{Bad}} & \thead{\textbf{Very}\\ \textbf{bad}} \\
	    \cmidrule(lr){3-7} \cmidrule(lr){10-15}
	    \textbf{16-24} && $145$ & $402$ & $84$ & $5$ & $3$ &&&& $98$ & $387$ & $83$ & $13$ & $3$\\
	    \textbf{25-34} && $112$ & $414$ & $74$ & $13$ & $2$ &&&& $108$ & $395$ & $90$ & $22$ & $4$\\ 
	    \textbf{35-44} && $80$ & $331$ & $82$ & $24$ & $4$ &&&& $67$ & $327$ & $99$ & $17$ & $4$\\
	    \textbf{45-54} && $54$ & $231$ & $102$ & $22$ & $6$ &&&& $36$ & $238$ & $134$ & $28$ & $10$\\
	    \textbf{55-64} && $30$ & $219$ & $119$ & $53$ & $12$ &&&& $23$ & $195$ & $187$ & $53$ & $18$\\
	    \textbf{65-74} && $18$ & $125$ & $110$ & $35$ & $4$ &&&& $26$ & $142$ & $174$ & $63$ & $16$\\
	    \textbf{+75} && $9$ & $67$ & $65$ & $25$ & $8$ &&&& $11$ & $69$ & $92$ & $41$ & $9$\\
	    \bottomrule
	    \end{tabular}
	    \caption{Data from the Spanish National Health Survey of 1997, see~\cite{greenacre1993Data}}
	    \label{tab:1}
	  \end{table}
	
	In~\cite{greenacre1993} data are reported from a survey over $6731$ people of both genders, from $16$ to over $75$ years old, who were asked to evaluate their health status. Then the answers were organized in a $3$-way table with dimensions $n_1 = 2$, $n_2 = 7$ and $n_3 = 5$, as in Table \ref{tab:1}. On mode~$1$ we have the two genders: male, `M', and female, `F' and on the second mode there are $7$ age groups. On the last mode we set $5$ health grades, from `Very Good' to `Very bad`. The $(1,1,1)$ entry of the multiway table is the number of men between $16$ and $24$ years old who judge `Very good' their health status. Let $\ten{F}\in\R^{n_1\times n_2\times n_3}$ represent the data relative frequencies in tensor format for MWCA. CA is realized over $A_3$ the matricization of $\ten{F}$ with respect to mode $3$, i.e. `health grade', where on the row we set the health categories and on the columns all the possible combinations of ages and gender.
	In Figure \ref{fig:2L} we recover the correspondence analysis of~\cite{greenacre1993}, while in \ref{fig:2R} we display multiway correspondence analysis with the first two columns of $Y_\mu = U_\mu \Sigma_{\mu}$, defined in Section~\ref{sec:5} for $\mu\in\{1,2,3\}$.
	\begin{figure}[h]
     \centering
    	\subfloat[Mode $1$ correspondence analysis.]{\includegraphics[scale =0.5,valign=t,width=0.48\textwidth]{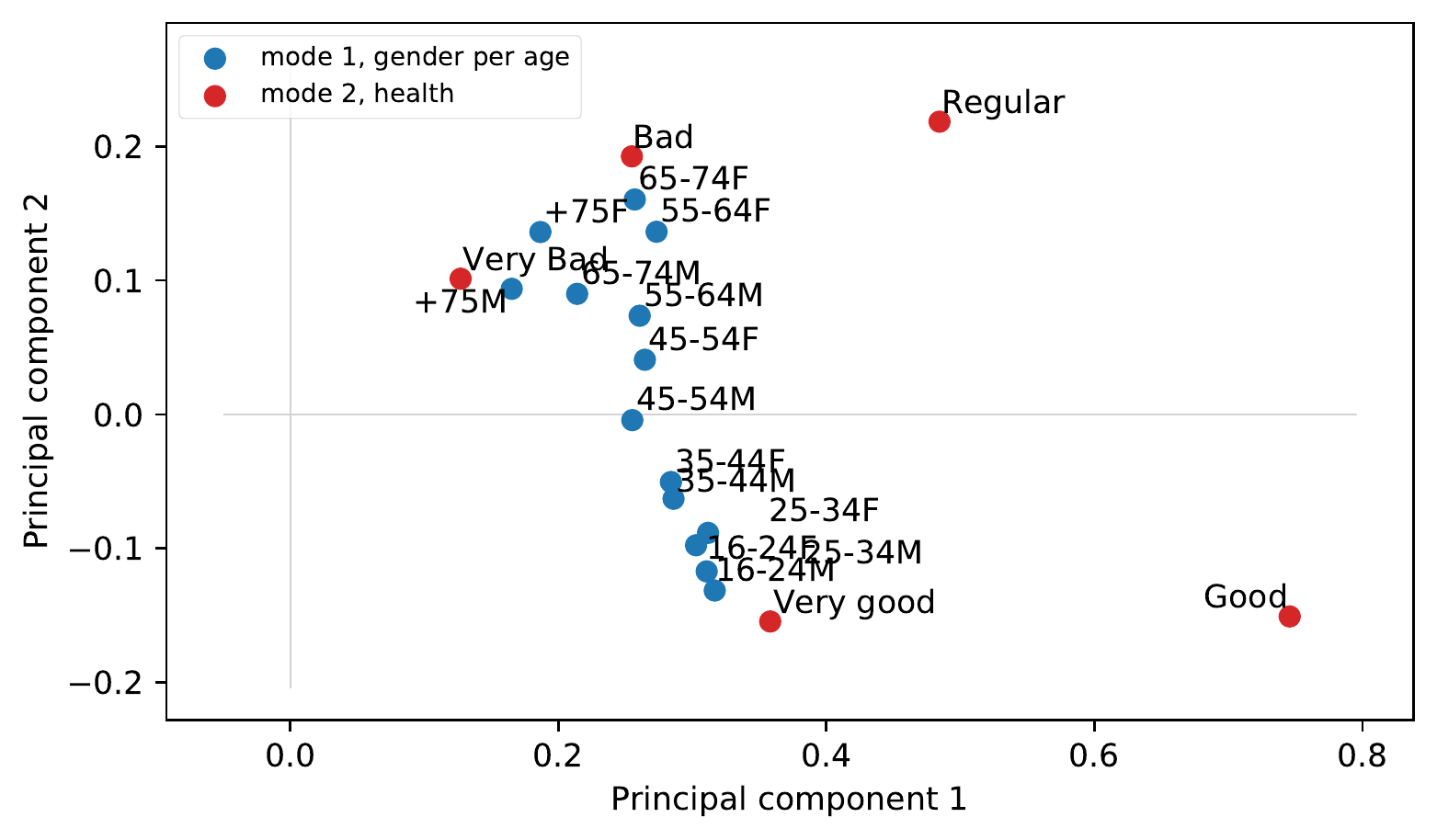}\label{fig:2L}}
    	\quad
    	\subfloat[Multiway correspondence analysis.]{\includegraphics[scale=0.5,valign=t,width=0.48\textwidth]{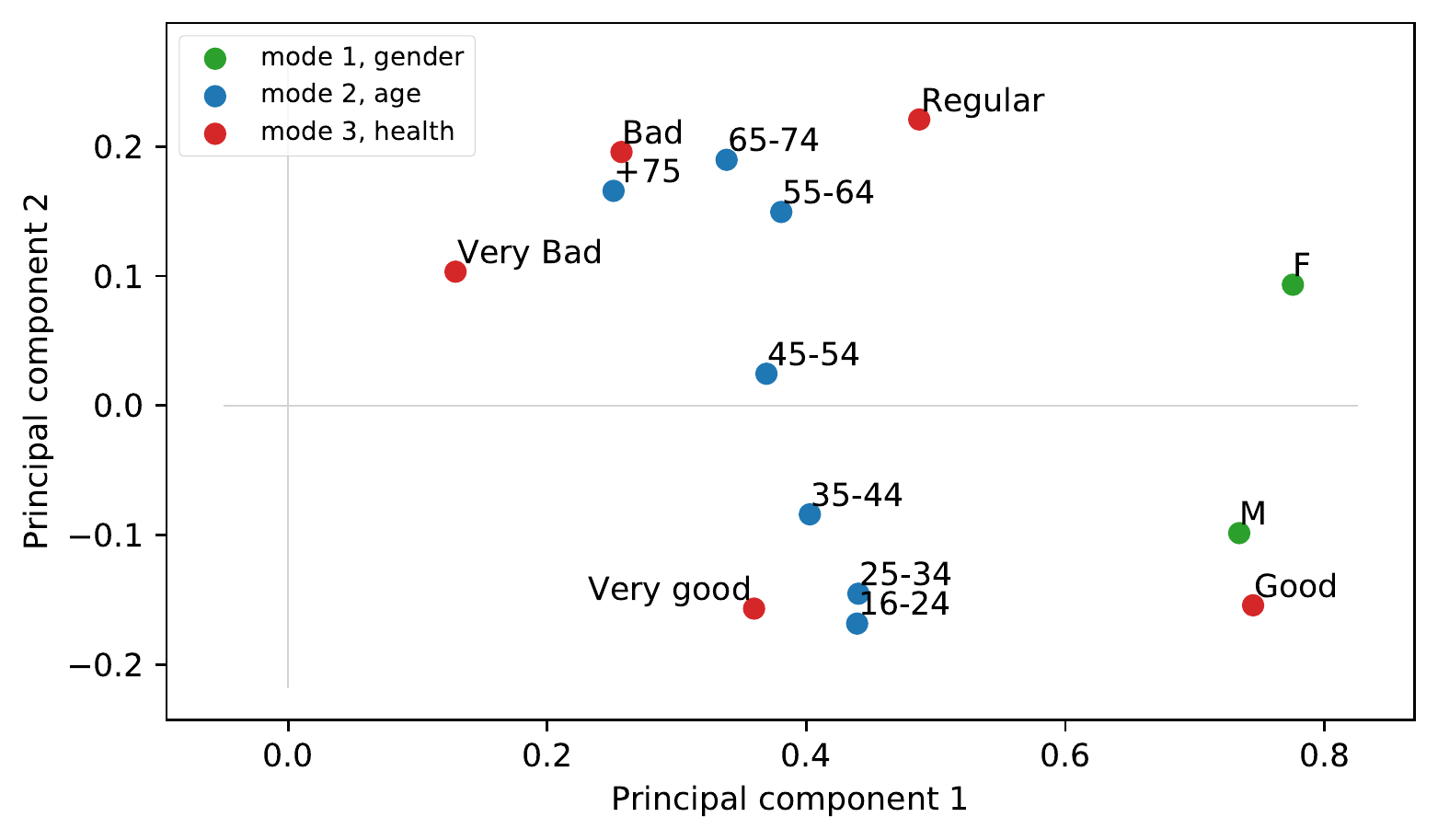}\label{fig:2R}}
     \caption{CA versus MWCA for data of~\cite{greenacre1993Data}.}
     \label{fig:2}
 	\end{figure}
	As in~\cite{greenacre1993}, both plots in Figure \ref{fig:2} display a gradient for the age categories: they distribute from the youngest in the bottom to the oldest near the top. The $5$ health categories are at the extremes of this gradient, the better health levels at the bottom, worst ones at the top. From this we infer that the health evaluation decreases for increasing ages. This phenomenon is highlighted by both the techniques, even if it is clearer in Figure \ref{fig:2R}, thanks to the possibility of analyzing the variables separately. Taking into account the gender in Figure \ref{fig:2R}, let us remark that the male dot is close to the `Good' one. Thanks to the barycentric relation, we know that the closer the points the more correlated they are. So men appear to provide optimistic evaluations of their health status. Women dot stands aside from the other, slightly closer to worst health status, suggesting that female health judging is balanced between the age classes, with a minor inclination toward pessimistic evaluation. In Figure \ref{fig:2L} for increasing ages, male and female dots tend to increase their distances, suggesting that when men and women age, they tend respectively to be more optimistic and pessimistic in evaluating their health status. MWCA plot seems to point out that men are more optimistic than how much pessimistic women are. In contrast from Figure \ref{fig:2R} we cannot conclude easily which age category presents the biggest difference in health evaluation for the two genders, information that can be inferred from \ref{fig:2L}. Lastly comparing the health point cloud in Figures \ref{fig:2L} and \ref{fig:2R}, we observe that their principal coordinates are almost the same and we explain it in terms of relative error from Equation~\eqref{eqn:errorCA-MCWA}. Indeed for this data-set the relative error $e(\tilde{{T}}^{(3)}, \tilde{A}_3) \approx 0.035$ is quite small.
	\end{example}
	
	\begin{example}
	The data-set provided by Malabar project (IFREMER, CNRS, INRAE, Labex COTE)~\cite{Malabardata} is a multiway contingency table of $4$-order. In this metabarcoding project, $32$ water samples have been collected in Arcachon Bay at four locations (Bouee13, Comprian, Jacquets, Teychan), during four seasons, and at two positions in water column (pelagic and benthic). DNA has been extracted, and Operational Taxonomic Units (OTUs) have been built. OTU are expected to correspond to species and, without entering too much into details here, the question which motivates Malabar project is to understand or quantify the role of location, season and water column in the diversity of protists. To address this question, a four way contingency table $\mathbf{T}$ has been built, where ${T}_{i,j,k,\ell}$ is the number of sequences collected at location $j$, in season $\ell$ and position $k$ which belong to OTU $i$. The size of the contingency tensor is $3539 \times 4 \times 2 \times 4$. Our contribution here is to make a first analysis of this data-set, comparing results provided by classical CA and MWCA in Figure \ref{fig:3}.  		 
	
	\begin{figure}
     \centering
    	\subfloat[Correspondence Analysis of mode $1$.]{\includegraphics[scale =0.5, valign=t, width=0.48\textwidth]{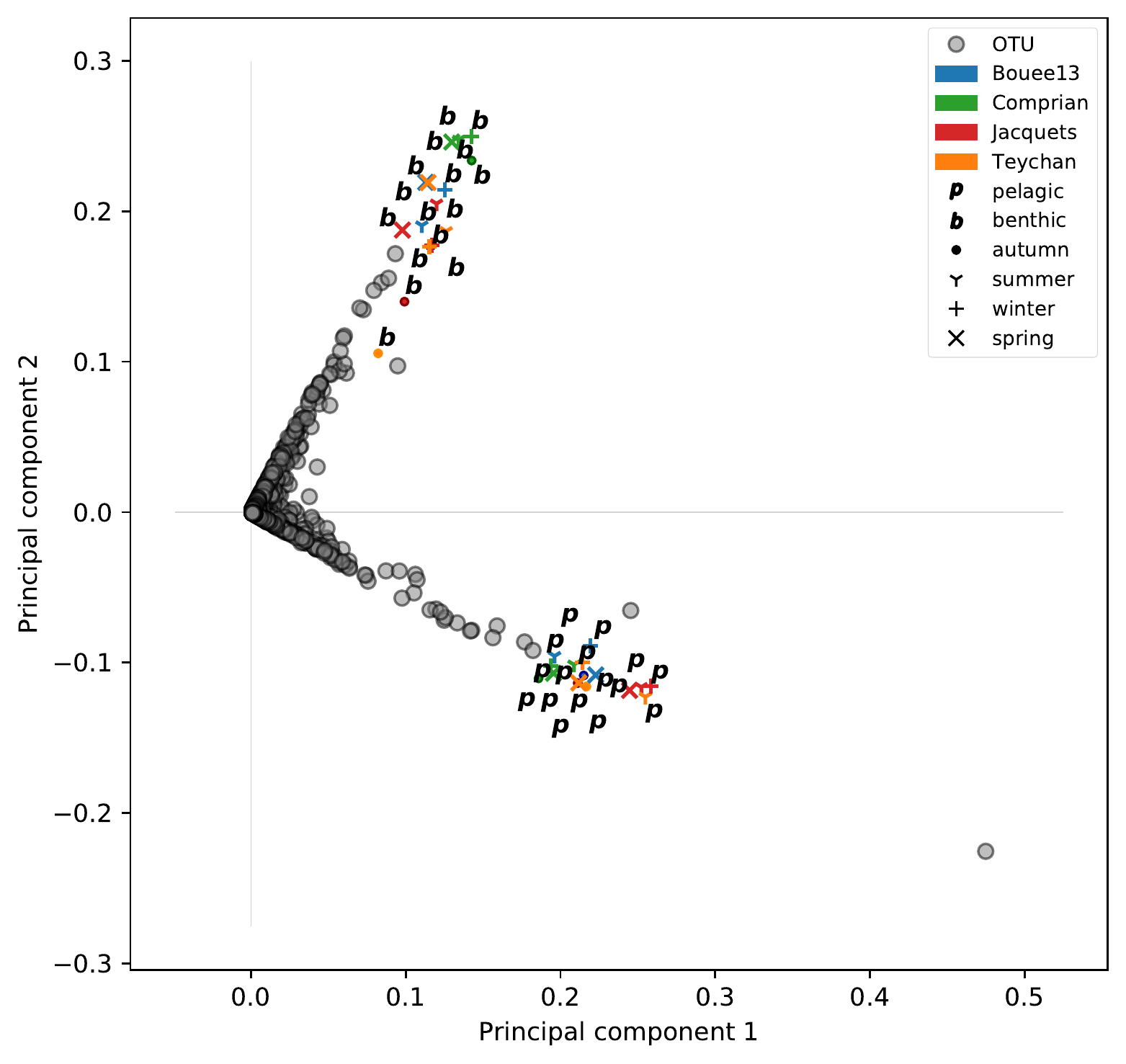}\label{fig:3aL}}
    	\quad
    	\subfloat[Multiway Correspondence Analysis.]{\includegraphics[scale=0.5, valign=t, width=0.48\textwidth]{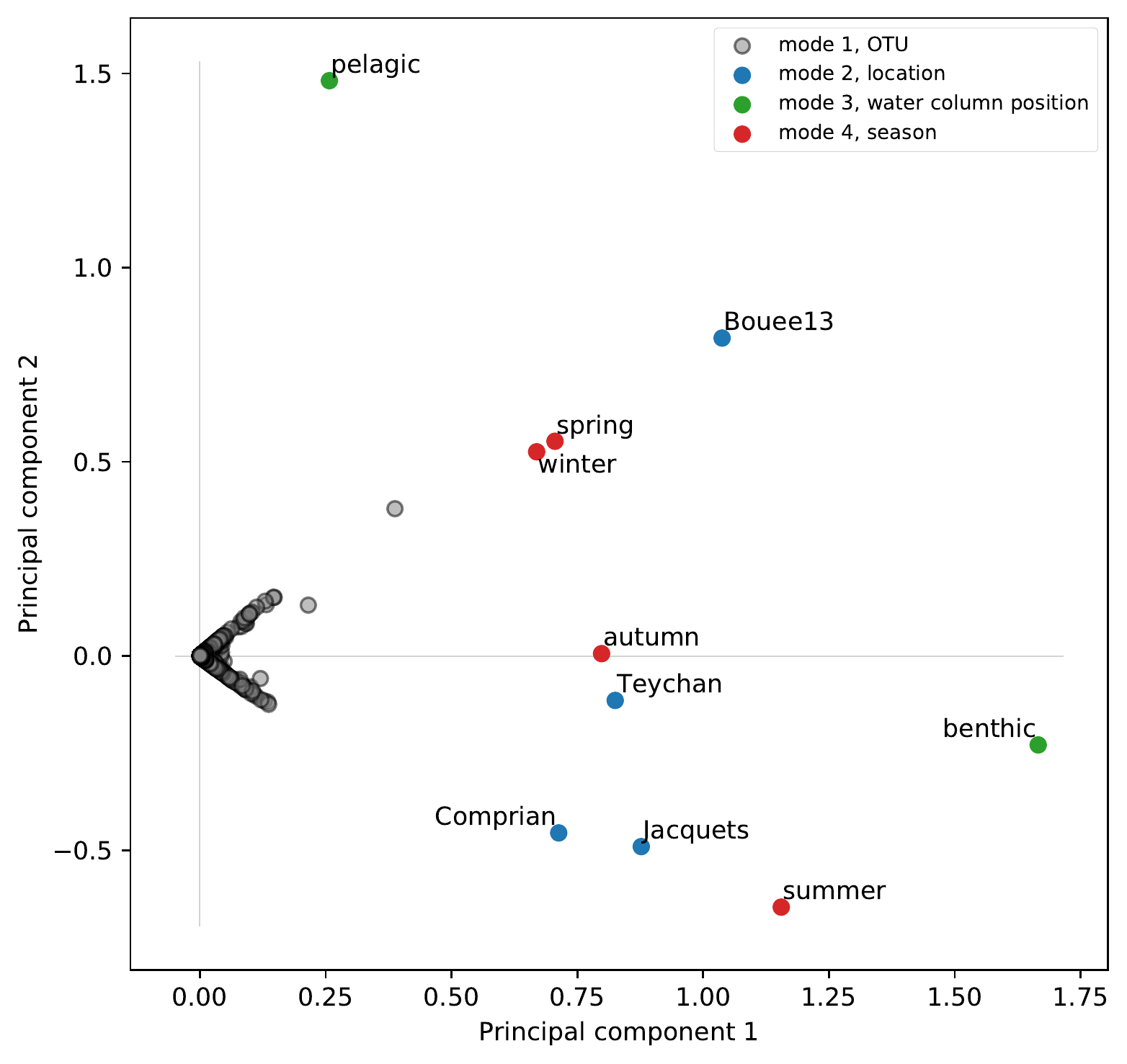}\label{fig:3aR}}
    	\vskip\baselineskip
    	\subfloat[CA with filtered OTU.]{\includegraphics[scale =0.5, valign=t, width=0.48\textwidth]{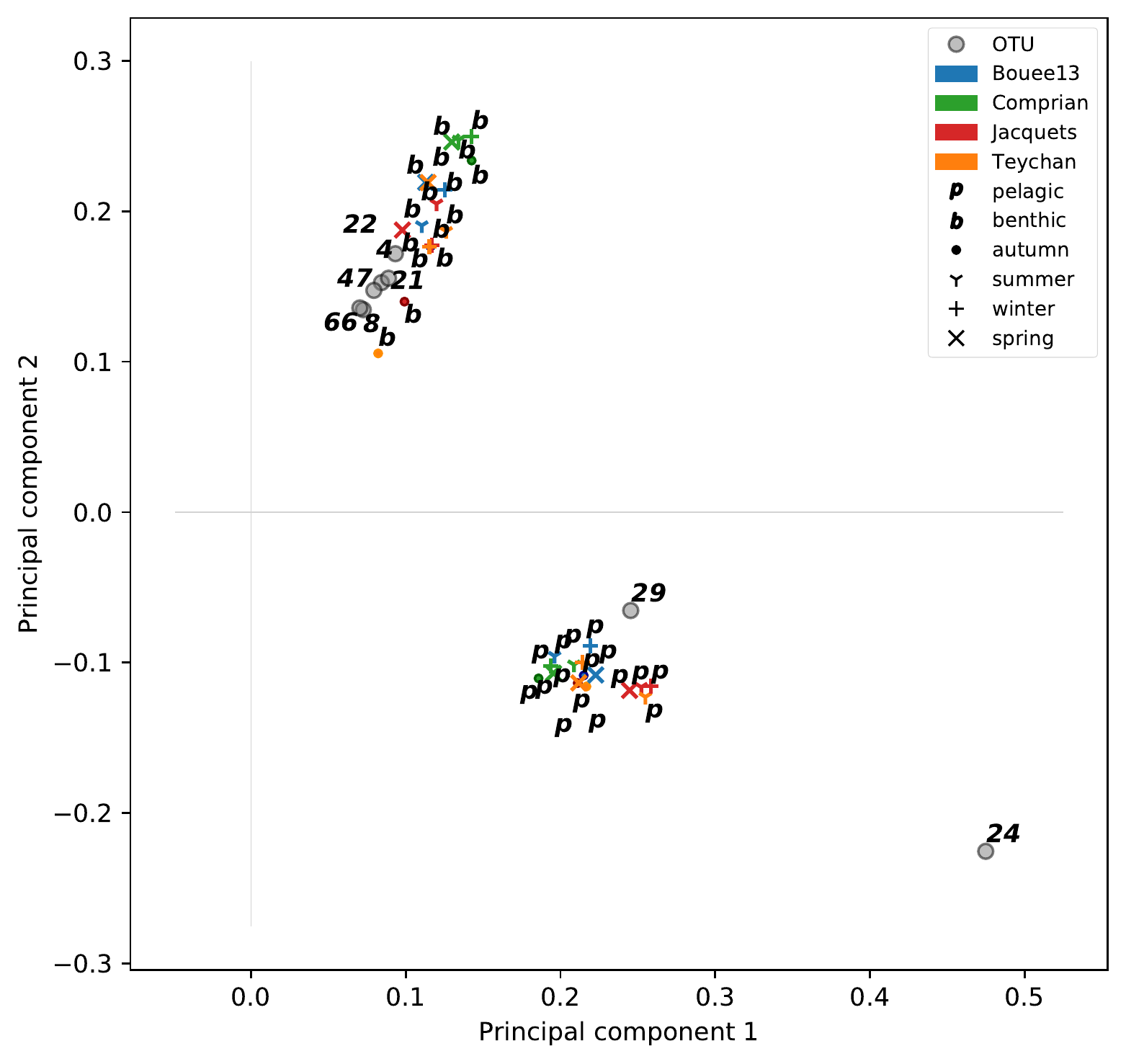}\label{fig:3bL}}
    	\quad
    	\subfloat[MWCA with filtered OTU.]{\includegraphics[scale=0.5, valign=t, width=0.48\textwidth]{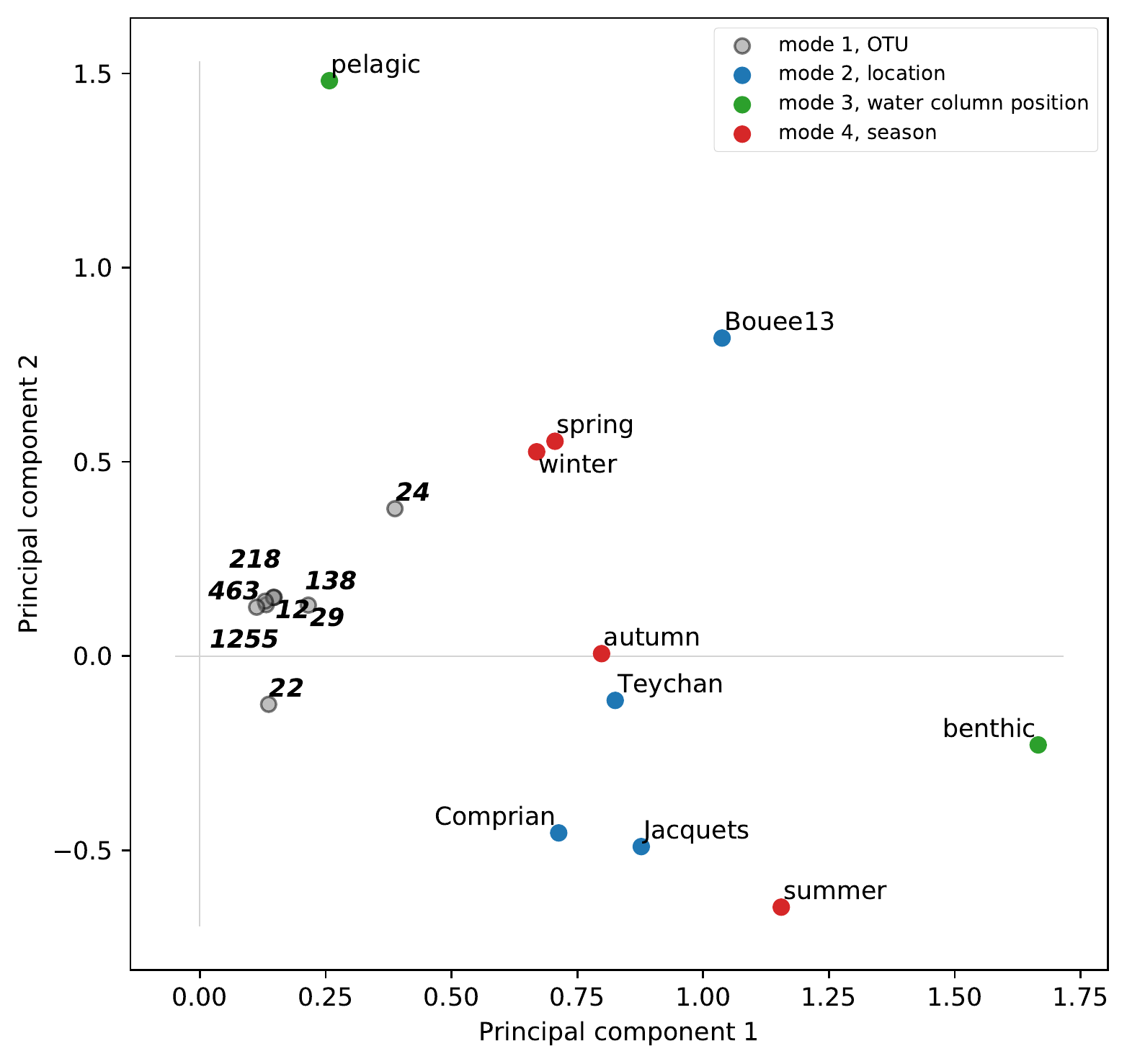}\label{fig:3bR}}
     \caption{CA vs MWCA biplot of Arcachon Bay data-set.}
     \label{fig:3}
	\end{figure}
  
  In Figure \ref{fig:3aL} the result of CA performed over $A_1$, the matricization of frequency tensor data $\ten{F}$ with respect to mode $1$, where on the rows there are the different OTUs and on the columns the locations, water column positions and seasons combined. Next to it, Figure \ref{fig:3aR} displays the result of MWCA over the tensor data. In the CA case OTU point cloud spreads over two orthogonal direction, led by the two positions in the water column. Similarly in MWCA plot the water column point cloud seem to organize and to affect the position and clustering of the others. Because of mode combination in Figure \ref{fig:3aL}, it is not evident which position in the water column affects which location or season. However thanks to the barycentric relation and the multiway approach, we can infer some relations among season, location and water column from Figure \ref{fig:3aR}, where they are independently displayed. Indeed pelagic point appears to affect the OTU present during spring and winter time. Similarly benthic position probably drives the distribution of OTU in autumn and summer. A similar argument can be repeated for location point cloud. Indeed the barycentric relation suggests that the benthic position leads the distribution of OTU in Comprian, Jacquets and Teychan, while the pelagic dot influences Bouee13 dot. The barycentric relation enables us to interpret the interaction among season and location point clouds. In Figure \ref{fig:3aR} winter and spring seem to be more correlated with Bouee13 and similarly summer and autumn are probably more correlated with the remaining three locations. Analysing the OTU and the season point clouds, we observe that winter and spring pull out one orthogonal direction over which part of the OTU points spread. Similarly summer pulls out the other OTU orthogonal direction. The autumn point appears to be neutral with respect to the OTU point cloud distribution. Similarly if we study the OTU and location point clouds interaction, one orthogonal direction is driven by Bouee13, one by Comprian and Jaquets, while Teychan seems to have a small effect on the OTU point distribution. 
  
  \par In Figures \ref{fig:3bL} and \ref{fig:3bR}, we filter the OTU by size, keeping just those with first principal components greater than $0.2$ and the second greater than $0.12$. Figure \ref{fig:3bL} confirms the idea of water-column position driving the OTU point cloud distribution. However in Figure \ref{fig:3aR} most of the OTU points are close to the origin. Indeed in this left plot we filtered $8$ OTU which are probably more correlated with benthic or pelagic condition. Comparing the two filtered figures, we see that only $3$ of the $8$ selected OTU are common between them, i.e., OTU `22', `24' and `29'. These common OTU share the same correlations between the two methods. Indeed in both case, `22' is more related to benthic position, while `24' and `29' with pelagic one. Moreover most of the OTU in Figure \ref{fig:3bR} are concentrated around the origin. We may infer that the action of the water column point cloud is softened by the other modes, here independent, leading to a more balanced distribution of OTU. In conclusion notice that in MWCA the different point clouds are more spread than in the CA case, enabling a deeper analysis of the variables interactions. It could be linked to the significant relative error $e(\tilde{{F}}^{(1)}, \tilde{A}_1) \approx 0.14$, which may also explain why CA and MWCA figures present a flipped $y$-axis. Indeed the mode combinations including benthic have positive y-components in Figure \ref{fig:3aL}, while benthic point has negative y-component in Figure \ref{fig:3aR}. However since it affects in the same way all the point clouds, the two methods lead to a coherent interpretation.
\end{example}

\section{Conclusion}

	The correspondence word in correspondence analysis comes from the relationship between the row and column point cloud coordinates, proven in the matrix case. We have proposed an extension of this correspondence to tensors by the High Order Singual Value Decomposition (HOSVD). Indeed, for each matricization of a given tensor, we computed the associated principal components. We first proved that each principal component is a linear combination of the Kronecker product of the others. Noting that, from the geometric viewpoint, principal components are the coordinates of point clouds, we provided a geometric interpretation of the algebraic link. The next step of our argument was to highlight the role of the metric in the principal component links. By finally choosing the CA metric, we recovered the barycentric link between the coordinates of the point clouds. The persistence of this relationship allows us to refer it as the extension of correspondence analysis to tensors through HOSVD. In Subsection \ref{sec2:5} we observed that several algorithms were proposed over the years to compute the Tucker basis. We focused on the simplest one, e.g., HOSVD. However the statements of all propositions can be proved using Tucker basis obtained from other algorithms, as, for example, HOOI,~\cite{Lathauwer2000b}. In this case a further algorithm to compute the singular values will be necessary to get the point clouds. Numerical experiments have shown two comparisons between the classical Correspondence Analysis (CA) performed on matricized tensor data and the MultiWay Correspondence Analysis (MWCA) implemented with HOSVD working with data structured as tensor. Both techniques guide to the same interpretation, although with different nuances. The main advantage MWCA is to highlight the global properties of a variable category. Indeed, the coordinates of each point cloud are computed independently from the others. Therefore MWCA catches at the same time the global pattern of point clouds and the minor ones. On the contrary CA is able to state clearly only the most important relation among two points clouds. It is worthwhile underlining that in CA combining variables helps in understanding their peculiarities. Indeed in CA two or more variables are coupled, because of the tensor matricization, and the analysis may highlight relations specific for a certain combination of modes. 
	\par In conclusion, the choice of MWCA or CA for interpreting multiway tables depends strongly on the objective of the data analysis one wishes to perform. We recommend MWCA when the objective is to discover an overall reciprocal correlation of different variables. While we suggest the CA when the objective is to identify specific combination patterns.

	\section{Acknowledgments.}
	The authors are very grateful to Luc Giraud for many helpful discussions and suggestions, which helped us improving this work. They thank also the teams of the Malabar project for sharing their data before their publication and letting them test their method.

\bibliographystyle{ieeetr}
\bibliography{references}

\begin{thebibliography}{10}

\bibitem{Izenman2008}
A.~J. Izenman, {\em Modern Multivariate Statistical Techniques}.
\newblock NY: Springer, 2008.

\bibitem{Wang2012}
J.~Wang, {\em Geometric structure if high-dimensional data and dimensionality
  reduction}.
\newblock Springer, 2012.

\bibitem{Pearson1901}
K.~Pearson, ``{LIII. O}n lines and planes of closest fit to systems of points
  in space,'' {\em The London, Edinburgh, and Dublin Philosophical Magazine and
  Journal of Science}, vol.~2, no.~11, pp.~559--572, 1901.

\bibitem{Hotelling1933}
H.~Hotelling, ``Analysis of a complex of statistical variables into principal
  components.,'' {\em Journal of Educational Psychology}, vol.~24, no.~6,
  pp.~417--441, 1933.

\bibitem{Gabriel1971}
K.~R. Gabriel, ``The biplot graphic display of matrices with application to
  principal component analysis,'' {\em Biometrika}, vol.~58, pp.~453--467, 12
  1971.

\bibitem{Atchley1975}
W.~Atchley and E.~Bryant, {\em Multivariate Statistical Methods: Among-Groups
  Covariation}.
\newblock Benchmark Papers in Systematic and Evolutionary Biology, Elsevier
  Science \& Technology Books, 1975.

\bibitem{Kendall1975}
M.~G. Kendall, {\em Multivariate analysis}.
\newblock Griffin London, 1975.

\bibitem{morrison1976}
D.~Morrison, D.~Morrison, L.~Marshall, M.-H. Incorporated, F.~Morrison,
  H.~Sahlin, N.~Y.~A. of~Sciences, L.~da~Vinci Academy~of Arts, and Sciences,
  {\em Multivariate Statistical Methods}.
\newblock Annals of the New York Academy of Sciences, McGraw-Hill, 1976.

\bibitem{chambers1977}
J.~Chambers, {\em Computational Methods for Data Analysis}.
\newblock Wiley Series in Probability and Mathematical Statistics, Wiley, 1977.

\bibitem{Jolliffe2002}
I.~T. Jolliffe, {\em Principal Component Analysis}.
\newblock Springer Series in Statistics, Springer-Verlag New York, second~ed.,
  2002.

\bibitem{Jolliffe2016}
I.~T. Jolliffe and J.~Cadima, ``Principal component analysis: a review and
  recent developments,'' {\em Philosophical Transactions of the Royal Society
  A: Mathematical, Physical and Engineering Sciences}, vol.~374, no.~2065,
  p.~20150202, 2016.

\bibitem{Mardia1979}
K.~V. Mardia, J.~Kent, and J.~M. Bibby, {\em Multivariate Analysis}.
\newblock Probability and Mathematical Statistics, Academic Press, 1979.

\bibitem{LMF82}
L.~Lebart, A.~Morineau, and J.-P. F\'enelon, {\em Traitement des donn\'ees
  statistiques}.
\newblock Dunod, Paris, 1982.

\bibitem{Saporta1990}
G.~Saporta, {\em Probabilit{\'e}s, Analyse de Données et Statistique}.
\newblock Editions Technip, 1990.

\bibitem{Anderson1958}
T.~Anderson, {\em An Introduction to Multivariate Statistical Analysis}.
\newblock Wiley Series in Probability and Statistics, Wiley, 2003.

\bibitem{Rao1973}
C.~R. Rao, {\em Linear statistical Infernece and its Applications}.
\newblock Wiley Series in Probability and Mathematical Statistics, Wiley,
  second~ed., 1973.

\bibitem{benzecri1973}
J.~Benz{\'e}cri and L.~Bellier, {\em L'analyse des donn{\'e}es: Benz{\'e}cri,
  J.-P. et al. L'analyse des correspondances}.
\newblock L'analyse des donn{\'e}es: le{\c{c}}ons sur l'analyse factorielle et
  la reconnaissance des formes, et travaux du Laboratoire de statistique de
  l'Universit{\'e} de Paris VI, Dunod, 1973.

\bibitem{Hill1974}
M.~O. Hill, ``Correspondence analysis: A neglected multivariate method,'' {\em
  Journal of the Royal Statistical Society: Series C (Applied Statistics)},
  vol.~23, no.~3, pp.~340--354, 1974.

\bibitem{Greenacre1977}
M.~Greenacre and L.~Degos, ``Correspondence analysis of {HLA} gene frequency
  data from 124 population samples,'' {\em American journal of human genetics},
  vol.~29, pp.~60--75, 02 1977.

\bibitem{greenacre1978}
M.~J. Greenacre, ``Some objective methods of graphical display of a data
  matrix,'' {\em Special Report). UNISA, Pretoria}, 1978.

\bibitem{gauch_1982}
H.~G. Gauch, {\em Multivariate Analysis in Community Ecology}.
\newblock Cambridge Studies in Ecology, Cambridge University Press, 1982.

\bibitem{gifi1990}
A.~Gifi, {\em Nonlinear Multivariate Analysis}.
\newblock Wiley Series in Probability and Statistics, Wiley, 1990.

\bibitem{Greenacre1984}
M.~J. Greenacre, {\em Theory and Applications of Correspondence Analysis}.
\newblock Academic Press, 1984.

\bibitem{Greenacre2007}
M.~J. Greenacre, {\em Correspondence Analysis in pratice}.
\newblock Chapman \& Hall/CRC, second~ed., 2007.

\bibitem{Cordier1965}
B.~Escofier-Cordier, {\em L'analyse factorielle des correspondances}.
\newblock PhD thesis, University of Rennes, 1965.

\bibitem{Benzecri1969}
J.~Benz{\'e}cri, ``Statistical analysis as a tool to make patterns emerge from
  data,'' in {\em Methodologies of Pattern Recognition} (S.~Watanabe, ed.),
  pp.~35--74, Academic Press, 1969.

\bibitem{Kroonenberg1983}
P.~M. Kroonenberg, {\em Three-mode principal component analysis: Theory and
  applications}.
\newblock DSWO Press, Leiden, 1983.

\bibitem{Bolasco1989}
R.~Coppi and S.~Bolasco, eds., {\em Multiway Data Analysis}.
\newblock North-Holland Publishing Co., 1989.

\bibitem{Franc1992}
A.~Franc, {\em Etude Alg\'ebrique des multitableaux: apports de l'alg\`{e}bre
  tensorielle}.
\newblock PhD thesis, University of Montpellier, 1992.

\bibitem{Lathauwer2000a}
L.~De~Lathauwer, B.~De~Moor, and J.~Vandewalle, ``A multilinear singular value
  decomposition,'' {\em SIAM Journal on Matrix Analysis and Applications},
  vol.~21, no.~4, pp.~1253--1278, 2000.

\bibitem{Lathauwer2000b}
L.~De~Lathauwer, B.~De~Moor, and J.~Vandewalle, ``On the best rank-1 and
  rank-$(r_1, r_2, \ldots, r_n)$ approximation of high order tensors,'' {\em
  SIAM Journal on Matrix Analysis and Applications}, vol.~21, no.~4,
  pp.~1324--1342, 2000.

\bibitem{Kroonenberg2008}
P.~M. Kroonenberg, {\em Applied Multiway Data Analysis}.
\newblock Wiley, 2008.

\bibitem{Kolda2009}
T.~G. Kolda and B.~W. Bader, ``Tensor decompositions and applications,'' {\em
  SIAM Review}, vol.~51, pp.~455--500, September 2009.

\bibitem{Franc1989}
A.~Franc, ``Multiway arrays: some algebraic remarks,'' in {\em Multiway Data
  Analysis} (R.~Coppi and S.~Bolasco, eds.), ch.~2, pp.~19--29, NLD:
  North-Holland Publishing Co., 1989.

\bibitem{Tucker1966}
L.~R. Tucker, ``Some mathematical notes on three-modes factor analysis,'' {\em
  Psychometrika}, vol.~31, no.~3, pp.~279--311, 1966.

\bibitem{SAND2006-2081}
T.~G. Kolda, ``Multilinear operators for higher-order decompositions,'' Tech.
  Rep. SAND2006-2081, Sandia National Laboratories, April 2006.

\bibitem{Kapteyn1986}
K.~Arie, H.~Neudecker, and T.~Wansbeek, ``An approach to $n-$mode components
  analysis,'' {\em Psychometrika}, vol.~51, no.~2, pp.~269--275, 1986.

\bibitem{Grasedick2010}
L.~Grasedyck, ``Hierarchical singular value decomposition of tensors,'' {\em
  SIAM Journal on Matrix Analysis and Applications}, vol.~31, no.~4,
  pp.~2029--2054, 2010.

\bibitem{StHOSVD}
N.~Vannieuwenhoven, R.~Vandebril, and K.~Meerbergen, ``A new truncation
  strategy for the higher-order singular value decomposition,'' {\em SIAM
  Journal on Scientific Computing}, vol.~34, no.~2, pp.~A1027--A1052, 2012.

\bibitem{Kroonenberg1980}
P.~Kroonenberg and J.~de~Leeuw, ``Principal component analysis of three modes
  data by means of alternating least square algorithms,'' {\em Psychometrika},
  vol.~45, no.~1, pp.~69--97, 1980.

\bibitem{Lastovicka1981}
J.~L. Lastovicka, ``The extension of component analysis to four-mode
  matrices,'' {\em Psychometrika}, vol.~46, no.~1, pp.~47--57, 1981.

\bibitem{BroSVD}
R.~Bro, E.~Acar, and T.~G. Kolda, ``Resolving the sign ambiguity in the
  singular value decomposition,'' {\em Journal of Chemometrics}, vol.~22,
  no.~2, pp.~135--140, 2008.

\bibitem{BroSignIndeter}
R.~Bro, R.~Leardi, and L.~G. Johnsen, ``Solving the sign indeterminacy for
  multiway models,'' {\em Journal of Chemometrics}, vol.~27, no.~3-4,
  pp.~70--75, 2013.

\bibitem{tensorly}
J.~Kossaifi, Y.~Panagakis, A.~Anandkumar, and M.~Pantic, ``Tensorly: Tensor
  learning in python,'' {\em Journal of Machine Learning Research}, vol.~20,
  no.~26, pp.~1--6, 2019.

\bibitem{greenacre1993Data}
CARME-N, ``Spanish national health survey.''
  \url{http://www.carme-n.org/?sec=data2}, 2007.
\newblock [Data-set 3].

\bibitem{Malabardata}
``Malabar{\ }project data paper ({IFREMER}, {CNRS}, {INRAE}, {L}abex {COTE}).''
  in prep., 202{X}.

\bibitem{greenacre1993}
M.~Greenacre, {\em Correspondence Analysis in Practice}.
\newblock Academic Press, 1993.
\newblock Data-sets are available at.

\end{thebibliography}
\end{document}